\documentclass{amsart}
\usepackage{amsmath}
\usepackage{amssymb}
\usepackage{amsfonts}
\usepackage{amsthm}
\usepackage{mathrsfs}
\usepackage[all]{xy}
\usepackage{ytableau}
\usepackage{color}
\usepackage{multicol}
\usepackage{enumerate}
\usepackage{enumitem}
\usepackage{comment}
\usepackage[vcentermath,enableskew]{youngtab}
\usepackage{multirow}
\usepackage{tikz-cd}
\usepackage{cancel}
\usepackage[margin=2.5cm]{geometry}

\usetikzlibrary{arrows}

    \usetikzlibrary{decorations.markings, arrows.meta}

    \tikzset{
        mid arrow/.style={
            postaction={decorate,decoration={
                markings,
                mark=at position .6 with {\arrow[scale=1.2, #1]{Stealth}}
            }}
        }
    }

\numberwithin{equation}{section}

\theoremstyle{theorem}

\newtheorem{theorem}{Theorem}[section]
\newtheorem{corollary}[theorem]{Corollary}
\newtheorem{lemma}[theorem]{Lemma}

\newtheorem{proposition}[theorem]{Proposition}

\theoremstyle{definition}

\newtheorem{remark}[theorem]{Remark}
\newtheorem{definition}[theorem]{Definition}

\newtheorem{example}[theorem]{Example}

\newcommand{\sub}{\subseteq}

\newcommand{\sm}{\setminus}

\newcommand{\CC}{\mathbb{C}}

\newcommand{\ZZ}{\mathbb{Z}}

\newcommand{\FF}{\mathbb{F}}

\newcommand{\ff}{\mathcal{F}}

\newcommand{\bb}{\mathcal{B}}
\newcommand{\dd}{\mathcal{D}}

\newcommand{\cc}{\mathcal{C}}
\newcommand{\mm}{\mathcal{M}}

\newcommand{\uu}{\mathcal{U}}
\newcommand{\oo}{\mathcal{O}}

\newcommand{\gl}{\mathfrak{gl}}
\newcommand{\sll}{\mathfrak{sl}}

\newcommand{\len}{\mathbf{len}}

\newcommand{\GL}{\operatorname{GL}}

\newcommand{\ol}{\overline}
\newcommand{\ra}{\rightarrow}

\newcommand{\onto}{\twoheadrightarrow}
\newcommand{\xra}{\xrightarrow[]}

\newcommand{\eps}{\epsilon}

\newcommand{\id}{\operatorname{id}}

\newcommand{\uvsln}{U_v(\sll_n)}

\newcommand{\drcomm}[1]{\textcolor{red}{#1}}

\allowdisplaybreaks

\usepackage[colorlinks=true, pdfstartview=FitV, linkcolor=blue, citecolor=blue, urlcolor=blue]{hyperref}
\usepackage{graphicx} 

\title{Representation theory of mirabolic quantum $\mathfrak{sl}_n$}
\author{Pallav Goyal}
\address{P.~Goyal: Department of Mathematics, University of California Riverside}
\email{pallavg@ucr.edu}

\author{Daniele Rosso}
\address{D.~Rosso: Department of Mathematics and Actuarial Science, Indiana University Northwest}
\email{drosso@iu.edu}

\date{\today}

\begin{document}

\begin{abstract}
We show that the mirabolic quantum group $MU(n)$ is a comodule algebra over the quantized enveloping algebra $U_v(\mathfrak{sl}_n)$, and use this structure to give a complete classification of its finite dimensional representations. In particular, we explicitly describe the construction of all irreducible finite dimensional representations of $MU(n)$ and show that the category of finite dimensional representations is semisimple. A crucial step involves constructing and analyzing Verma-type universal representations of $MU(n)$.
\end{abstract}

\maketitle

\section{Introduction}

In a classic paper from 1990 (\cite{BLM}), Beilison, Lusztig and MacPherson defined a convolution product on the space of $\GL_d$-invariant functions over the variety of pairs of \emph{$n$-step partial flags} in a $d$-dimensional vector space over the finite field $\FF_q$. This gave a geometric construction of the quantum Schur algebras, which are finite dimensional quotients of the quantized enveloping algebra of the Lie algebra $\gl_n$ (or $\sll_n$). They also obtained the idempotented version of the quantized enveloping algebra via a stabilization procedure for $d\to\infty$.

The construction from \cite{BLM} can be generalized by considering spaces of partial flags for Lie or algebraic groups other than the general linear group $\GL_d$. By considering pairs of isotropic partial flags, with the obvious action of an orthogonal or symplectic group, analogues in other types can be obtained. In \cite{BKLW} the authors considered groups of type B and C, and they constructed certain coideal subalgebras of the quantized enveloping algebra of $\gl_n$, which are sometimes called $\iota$quantum groups. Similarly, \cite{FL15} studied the case of groups of type $D$.

In another direction, the construction of \cite{BLM} can be extended to the `mirabolic' setting, as follows. Instead of considering pairs of flags, 
we can consider triples of two flags and a vector in $\FF_q^d$, where there are still finitely many orbits for the $\GL_d$ action (independently of $\FF_q$). A convolution product was defined by the second author in \cite{R15}, obtaining a mirabolic version of the quantized enveloping algebra. In particular, a complete classification of the finite dimensional representations for $MU(2)$ (the mirabolic quantum group version of $\sll_2$) was obtained. 

Recently, \cite{FZM} expanded on the work of \cite{R15}: they obtained explicit multiplication formulae and a description in terms of generators and relations for the mirabolic quantum group corresponding to $\gl_n$, 
also describing a stabilization procedure for $d\to\infty$, analogous to \cite{BLM}.

In this paper, we take advantage of the presentation given in \cite{FZM} to give a complete classification of the irreducible finite dimensional representations of mirabolic quantum $\sll_n$ and show that the category of finite dimensional representations is semisimple.

A summary of our main results and the organization of the paper is as follows:
\begin{itemize}
    \item In Section \ref{sec:mir-group-rep}, we give the necessary definitions and background on the mirabolic quantum group $MU(n)$ and show that it has a comodule algebra structure over $\uvsln$. We then use this structure in Section~\ref{sec:mirareps} to provide constructions of irreducible $MU(n)$-representations $L_{\lambda,r}$ that are parametrized by pairs $(\lambda,r)$ of a dominant integral weight $\lambda$ for $\sll_n$ and an integer $0\leq r\leq n$. 
    \item In Section \ref{sec:proof-prop}, we introduce the notion of the depth of an $MU(n)$-representation (which precisely corresponds to $r$ for the above representations). This is the main new algebraic property that distinguishes the categories of finite dimensional $MU(n)$-representations and $\uvsln$-representations.
    \item In Section~\ref{sec:Verma}, we give two constructions of (infinite dimensional) representations of $MU(n)$ that are the mirabolic analogues of the Verma modules for $\uvsln$. One of the main results in this section is Theorem~\ref{prop:Verma}, which says that both of our constructions give the same representation up to isomorphism. We use this theorem to provide a presentation for the irreducible representations $L_{\lambda, r}$ in terms of generators and relations in Corollary~\ref{cor:exppresent}.
    \item In Section \ref{sec:ss}, we show that every irreducible finite dimensional representation of $MU(n)$ is isomorphic to $L_{\lambda, r}$ for some $(\lambda, r)$ (Theorem~\ref{thm:class}) and that the category of finite dimensional representations for $MU(n)$ is semisimple (Theorem~\ref{thm:ss}). As a consequence of the classification, we show in Proposition~\ref{prop:bmc} that the category of these representations is a braided module category over the category of finite dimensional $\uvsln$-representations.
\end{itemize}

In a follow-up paper, we will use this classification to give an explicit combinatorial description of the mirabolic quantum Schur-Weyl duality discussed in \cite{R15} and \cite{FZM}.

\subsection{Future Directions}
Our work raises several interesting questions, which can be the basis for future research.
\begin{itemize}
\item The classification of representations and semisemplicity here are proven in the generic case, but these results will not hold if we allow $v$ to be a root of unity. It should be possible to use what is known about $\uvsln$ at roots of unity to investigate the representation theory of the mirabolic quantum group in that setting.
\item In this paper, we have worked in the category of finite dimensional representations, but it should also be interesting to extend the study of representations to some analogue of category $\oo$. The results in Section~\ref{sec:Verma} would be a starting point for such investigations.
\item There is a potentially interesting structure coming from induction and restriction functors between representations over $MU(n)$ and $MU(n+1)$ that can be studied, in analogy to the Gelfand-Zeitlin results for Lie algebras/quantum groups of type A.
\item The construction of \cite{BLM} was generalized to affine Type A in \cite{GV93}, and to other affine classical types in \cite{FLLLW20}. A construction and some results for a mirabolic affine setting were discussed in \cite{RPhD}, and one would expect to find extensions of our results in that case too.
\end{itemize}

\subsection*{Acknowledgments}
D.R. was supported in this research by an AMS-Simons Research Enhancement Grant for PUI faculty.
\section{The mirabolic quantum group and its representations}\label{sec:mir-group-rep}
\subsection{Background and Notation}
We work over the field $K=\CC(v)$, where $v$ is an indeterminate. Let $A=(a_{ij})$ be the Cartan matrix of type $A_{n-1}$, that is, for all $1\leq i,j,\leq n-1$ we have
$$a_{ij}=\begin{cases} 2 & \text{ if }i=j, \\ -1 & \text{ if }|i-j|=1, \\ 0 & \text{ if }|i-j|>1.\end{cases}$$

\begin{definition}\label{def:uvsln}
The quantum enveloping algebra of the simple Lie algebra $\sll_n$, denoted by $\uvsln$, is the unital associative algebra over $K$ with generators $e_i, f_i, k_i^{\pm1}$ for $1\leq i \leq n-1$, subject to the relations:
\begin{align}\label{eq:usln.1} k_i^{\pm1}k_i^{\mp1}=1, & \quad k_ik_j=k_jk_i \\
\label{eq:usln.2} k_ie_jk_i^{-1} = v^{a_{ij}}e_j, & \quad k_if_jk_i^{-1} = v^{-a_{ij}}f_j\\
\label{eq:usln.3} e_if_j-f_je_i & =  \delta_{ij}\frac{k_i-k_i^{-1}}{v-v^{-1}}\\
\label{eq:usln.4} e_ie_j=e_je_i, & \quad f_if_j=f_jf_i \quad\text{ if }|i-j| \neq 1 \\
\label{eq:usln.5} (v+v^{-1})e_ie_je_i & =  e_i^2e_j + e_je_i^2 \quad\text{ if }|i-j| = 1\\
\label{eq:usln.6} (v+v^{-1})f_if_jf_i & =  f_i^2f_j + f_jf_i^2 \quad\text{ if }|i-j| = 1.\end{align}
\end{definition}

\begin{remark}
It is well known that $\uvsln$ is a Hopf algebra, with comultiplication $\Delta:\uvsln\to \uvsln\otimes \uvsln$ given on the generators by 
\[\Delta(e_i) = 1\otimes e_i + e_i\otimes k_i,\quad\Delta(f_i) = k_i^{-1}\otimes f_i + f_i\otimes 1,\quad\Delta(k_i^{\pm 1}) = k_i^{\pm 1}\otimes k_i^{\pm 1},\]
counit $\varepsilon:\uvsln\to K$ given by
\[\varepsilon(k_i^{\pm 1})=1,\qquad \varepsilon(e_i)=\varepsilon(f_i)=0,\]
and antipode $S:\uvsln\to \uvsln$ given by:
\[S(k_i^{\pm 1})=k_i^{\mp1},\qquad S(e_i)=-k_i^{-1}e_i, \qquad S(f_i)=-f_ik_i,\]
for all $1\leq i\leq n-1$
\end{remark}

The main object of study of this paper is the following algebra.

\begin{definition}\label{def:mir-sln}
The \emph{mirabolic quantum group}, denoted by $MU(n)$, is the unital associative algebra over $K$ with generators $\ell$, and $e_i, f_i, k_i^{\pm 1}$ for $1\leq i\leq n-1$,  subject to the relations of Definition \ref{def:uvsln}, plus the additional ones involving $\ell$:
\begin{align}\label{eq:mun.1} \ell^2=\ell,&\quad k_i\ell=\ell k_i\\
\label{eq:mun.2} \ell e_i=e_i\ell,&\quad \ell f_i=f_i\ell \quad\text{ if }i\geq 2\\
\label{eq:mun.3} \ell e_1\ell=\ell e_1,&\quad \ell f_1\ell=f_1\ell\\
\label{eq:mun.4} (v+v^{-1})e_1\ell e_1 &= v^{-1}e_1^2\ell + v\ell e_1^2\\
\label{eq:mun.5} (v+v^{-1})f_1\ell f_1 &= v^{-1}\ell f_1^2 + vf_1^2\ell.\end{align}
\end{definition}
\begin{remark}\label{rem:grading}
The algebra $MU(n)$ is $\ZZ^{n-1}$-graded by setting, for all $1\leq i\leq n-1$,
$$\deg e_i=(0,\ldots,0,\stackrel{i}{1},0,\ldots,0),\quad \deg f_i=(0,\ldots,0,\stackrel{i}{-1},0,\ldots,0),$$
$$\deg k_i=\deg k_i^{-1}=\deg \ell=(0,\ldots,0).$$
\end{remark}
\begin{remark}The definition of this algebra comes from a convolution construction. We recall some of the basic notions here and we refer the reader to \cite[\S 2]{R15} for more details. We let $\FF_q$ be the finite field with $q$ elements, and $\ff(n,d)$ be the variety of $n$-step partial flags in $\FF_q^d$. Consider the diagonal action of $\GL(\FF_q^d)$ on $\ff(n,d)\times\ff(n,d)\times \FF_q^d$ which has finitely many orbits parametrized by decorated matrices. The space of invariant functions $\mm\uu_q(n,d):=\CC\left(\ff(n,d)\times\ff(n,d)\times \FF_q^d\right)^{\GL(\FF_q^d)}$ is a unital associative algebra with the convolution product defined by
\begin{equation} \label{eq:convo}
(\alpha*\beta)(F,F',u):=\sum_{F''\in\ff(n,d),~u'\in\FF_q^d}\alpha(F,F'',u')\beta(F'',F',u-u'),
\end{equation}
which we call the \emph{mirabolic quantum Schur algebra}. Since the structure constants for $\mm\uu_q(n,d)$ are polynomials in $q$, we can consider it to be the specialization of the \emph{generic mirabolic quantum Schur algebra} $MU(n,d)$, which is defined over $\CC(v)$, where $q=v^2$. By \cite{FZM}, the generators of $MU(n,d)$ satisfy the relations in Definitions \ref{def:uvsln} and \ref{def:mir-sln} (in addition to other relations that depend on $d$) hence we have a surjective map $MU(n)\onto MU(n,d)$ for all $d$. We notice here that in \cite{FZM} the authors prefer to work with the `general linear' analogue of the mirabolic quantum group, while we are using the `special linear' version, which is why the `Cartan part' has $n$ generators in their case ($H_a$, $1\leq a\leq n$, see \cite[Def. 2.1.2]{FZM}), and $n-1$ in our case ($k_i$, $1\leq i\leq n-1$). This difference is not significant for our purposes.
\end{remark}
\begin{remark}\label{rem:anti-inv}
There are two anti-involutions $\star$ and $\dagger$ on $MU(n)$ defined on the generators by 
\[\ell^\star=\ell,\quad (k_i^{\pm 1})^\star=k_i^{\pm 1},\quad e_i^\star=f_i,\quad f_i^\star=e_i.\]
\[\ell^\dagger=\ell,\quad (k_i^{\pm 1})^\dagger=k_i^{\pm 1},\quad e_i^\dagger=v^{-1}k_i^{-1}f_i=vf_ik_i^{-1},\quad f_i^\dagger=ve_ik_i=v^{-1}k_ie_i.\]
\end{remark}
\begin{proposition}\label{prop:comodule-alg}
\begin{enumerate}
\item There are two surjective algebra maps $\pi_0,\pi_1:MU(n)\to \uvsln$ that map the generators $e_i, f_i, k_i, k_i^{-1}$ to the ones of the same name, and $\pi_0(\ell)=0$, $\pi_1(\ell)=1$ respectively.
\item There is an injective algebra map $\iota:\uvsln\to MU(n)$ that maps the generators $e_i, f_i, k_i, k_i^{-1}$ to the same ones for $MU(n)$. 
\item The algebra $MU(n)$ is a co-module algebra for $\uvsln$ via the map:
\[\rho: MU(n) \ra \uvsln\otimes MU(n) ,\]
defined on generators via:
\[\rho(e_i)=\Delta(e_i) = 1\otimes e_i + e_i\otimes k_i, \qquad \rho(f_i)=\Delta(f_i) = k_i^{-1}\otimes f_i + f_i\otimes 1,\]
\[\rho(k_i^{\pm1})=\Delta(k_i^{\pm 1}) = k_i^{\pm 1}\otimes k_i^{\pm 1}, \qquad \rho(\ell)=1\otimes \ell.\]
\end{enumerate}
\end{proposition}
\begin{proof}
\begin{enumerate}
\item Since the generators $e_i, f_i, k_i, k_i^{-1}$ satisfy the same relations in $MU(n)$ as they do in $\uvsln$, and since replacing $\ell$ with either $1$ or $0$ satisfies the relations of Definition \ref{def:mir-sln}, it is clear that both maps are well defined. They are clearly surjective since the generators of $\uvsln$ are in the image of both maps.
\item It is clear that the map is well defined, since the generators of $\uvsln$ satisfy the same relations in $MU(n)$. Notice that the composition $\pi_0\circ\iota=\id_{\uvsln}$, hence $\iota$ is injective. 
\item It is clear that all the axioms of the coaction are satisfied on the generators $e_i$, $f_i$, $k_i^{\pm 1}$, since on those the coaction is equal to the comultiplication in $\uvsln$. Then, we have 
$(\Delta\otimes \id)\rho(\ell)=1\otimes 1\otimes \ell=(\id\otimes \rho)\rho(\ell)$, and $(\varepsilon\otimes \id)\rho(\ell)=\ell$.
To conclude, we check that $\rho$ satisfies the relations in Definition \ref{def:mir-sln}.
\[\rho(\ell^2)=\rho(\ell)\rho(\ell)=(1\otimes \ell)(1\otimes \ell)=1\otimes \ell^2=1\otimes \ell=\rho(\ell).\]
For all $i=1,\ldots,n-1$, we have:
\[\rho(k_i\ell)=(k_i\otimes k_i)(1\otimes \ell)=(k_i\otimes k_i\ell)=(k_i\otimes \ell k_i)=(1\otimes \ell)(k_i\otimes k_i)=\rho(\ell k_i).\]
For all $i=2,\ldots,n-1$, we have:
\[\rho(\ell e_i)=(1\otimes \ell)(1\otimes e_i+e_i\otimes k_i)=1\otimes \ell e_i+e_i\otimes \ell k_i=1\otimes e_i\ell+e_i\otimes k_i\ell=(1\otimes e_i+e_i\otimes k_i)(1\otimes \ell)=\rho(e_i\ell); \]
\[\rho(\ell f_i)=(1\otimes \ell)(k_i^{-1}\otimes f_i+f_i\otimes 1)=k_i^{-1}\otimes \ell f_i+f_i\otimes \ell=k_i^{-1}\otimes f_i\ell+f_i\otimes \ell=(k_i^{-1}\otimes f_i+f_i\otimes 1)(1\otimes \ell)=\rho(f_i\ell).\]
Finally,
\begin{multline*}\rho(\ell e_1\ell)=(1\otimes \ell)(1\otimes e_1+e_1\otimes k_1)(1\otimes \ell)=1\otimes \ell e_1\ell+e_1\otimes \ell k_1\ell\\=1\otimes \ell e_1+e_1\otimes \ell^2k_1=1\otimes \ell e_1+e_1\otimes \ell k_1=(1\otimes \ell)(1\otimes e_1+e_1\otimes k_1)=\rho(\ell e_1) \end{multline*}
\begin{multline*}\rho(\ell f_1\ell)=(1\otimes \ell)(k_1^{-1}\otimes f_1+f_1\otimes 1)(1\otimes \ell)=k_1^{-1}\otimes \ell f_1\ell+f_1\otimes \ell^2\\=k_1^{-1}\otimes f_1\ell+f_1\otimes \ell=(k_1^{-1}\otimes f_1+f_1\otimes 1)(1\otimes \ell)=\rho(f_1\ell) \end{multline*}
\begin{align*}
\rho((v+v^{-1})e_1\ell e_1)&= (v+v^{-1})(1\otimes e_1+e_1\otimes k_1)(1\otimes \ell)(1\otimes e_1+e_1\otimes k_1) \\
&= (v+v^{-1})(1\otimes e_1\ell e_1+e_1\otimes k_1\ell e_1+e_1\otimes e_1\ell k_1+e_1^2\otimes k_1\ell k_1)\\
&= 1\otimes (v+v^{-1}) e_1\ell e_1+v\otimes e_1\otimes k_1\ell e_1+v^{-1}e_1\otimes k_1\ell e_1+ \\
&~+ ve_1\otimes e_1\ell k_1+v^{-1}e_1\otimes e_1\ell k_1+ve_1^2\otimes k_1\ell k_1+v^{-1}e_1^2\otimes k_1\ell k_1\\
&= 1\otimes (v^{-1}e_1^2\ell+v\ell e_1^2)+v\otimes e_1\otimes \ell k_1e_1+v^{-1}e_1\otimes v^2 \ell e_1k_1+ \\
&~+ ve_1\otimes v^{-2}k_1e_1\ell+v^{-1}e_1\otimes e_1k_1\ell+ve_1^2\otimes \ell k_1^2+v^{-1}e_1^2\otimes k_1^2\ell \\
&= v^{-1}(1\otimes e_1^2\ell)+v^{-1}e_1\otimes k_1e_1\ell+v^{-1}e_1\otimes e_1k_1\ell+v^{-1}e_1^2\otimes k_1^2\ell\\
&~+ v(1\otimes \ell e_1^2)+v e_1\otimes \ell k_1e_1+v e_1\otimes \ell e_1k_1+ve_1^2\otimes \ell k_1^2 \\
&= v^{-1}(1\otimes e_1^2+e_1\otimes k_1e_1+e_1\otimes e_1k_1+e_1^2\otimes k_1^2)(1\otimes \ell) \\
&~+v(1\otimes \ell)(1\otimes e_1^2+ e_1\otimes k_1e_1+ e_1\otimes e_1k_1+e_1^2\otimes k_1^2) \\
&= v^{-1}(1\otimes e_1+e_1\otimes k_1)^2(1\otimes \ell)+v(1\otimes \ell)(1\otimes e_1+e_1\otimes k_1)^2\\
&= \rho(v^{-1}e_1^2\ell+v\ell e_1^2).
\end{align*}
The computation for $f_1\ell f_1$ is similar to the one for $e_1\ell e_1$.
\end{enumerate}    
\end{proof}
We end this section with a couple of lemmas that will be useful for computations later.

\begin{lemma} \label{lem:reduce}
Let $P(e)$ be any non-commutative polynomial in the $e_i$'s and $Q(f)$ be any non-commutative polynomial in the $f_j$'s. Then,
\[\ell P(e) \ell = \ell P(e),\]
\[\ell Q(f) \ell = Q(f)\ell.\]
\end{lemma}

\begin{proof}
This is an immediate consequences of the defining relations \eqref{eq:mun.2} and \eqref{eq:mun.3}.
\end{proof}
\begin{lemma} \label{lem:std}
Consider an element of the form $m=\ell Q(f)$ in $MU(n)$, where $Q(f)$ is some non-commutative monomial in the $f_i$'s. Suppose $Q(f)$ is not of the form $f_1f_{2}\cdots f_{s-1} f_s$ for some $s$. Then $m$ must be equal to a linear combination of monomials of the form $P'(f) \ell Q'(f)$ (having the same length as $m$ in terms of the generators) such that the monomials $Q'(f)$ are smaller in length than $Q(f)$.
\end{lemma}

What the above lemma essentially says is that if the monomial $Q$ is not of the specified form, then it is possible to move one or more of the $f_j$'s to the left of the $\ell$.

\begin{proof}
It suffices to prove the claim when $Q(f)$ is of the form $f_1  f_2 \cdots f_{k-1} f_k f_s$, where $s\neq k+1$.

We first treat the cases $k=0$ and $k=1$. If $k=0$, we just get that $m= \ell f_s= f_s \ell$ as $s\neq 1$, and we're done. If $k=1$, we get $m=\ell f_1f_s$. If $s>2$, we can commute the $f_s$ past both $f_1$ and $\ell$. If $s=1$, we get $m=\ell f_1^2 = (v^2+1)f_1 \ell f_1 - v^2 f_1^2\ell$, which is of the required form.

Now, suppose $k>1$. We use induction on $k$. If $s>k+1$, we can commute $f_s$ to the left of all the other elements in $m$ to get the required claim. If $s<k-1$, we can commute $f_s$ past $f_k$ and be done by induction. The cases that remain are $s=k$ and $s=k-1$.

If $s=k$, we write:
\[m=\ell f_1 f_{2}\cdots f_{k-1} f_k^2.\]
By using $f_{k-1} f_{k}^2 = (v+v^{-1})f_kf_{k-1} f_k - f_k^2f_{k-1}$, we note that we can commute at least one copy of $f_k$ to the left of $f_{k-1}$ in both the above terms, following which we can commute the $f_k$ past all other terms, and eventually $\ell$, to get the claim. 

If $s=k-1$, we write:
\[m=\ell f_{1} f_2\cdots f_{k-1}f_k f_{k-1}.\]
First, we use $(v+v^{-1})f_{k-1}f_k f_{k-1} = f_{k-1}^2 f_k + f_{k}f_{k-1}^2$. For the second term involving $f_k f_{k-1}^2,$ we can commute $f_k$ past all other generators to its left and get the required expression. For the first, we use induction as we are reduced to a case where $k$ is replaced by $k-1$ and $s=k-1$. This completes the proof.
\end{proof}
\begin{remark}
    An analogue Lemma \ref{lem:std} is also true for expressions of the form $Q(f) \ell$, where we can commute an $f_i$ to the right of $\ell$ as long as $Q(f)$ is not of the form $Q(f)=f_s f_{s-1}\cdots f_2 f_1$ for some $s$. 
    Lemma \ref{lem:std} is also true if we replace the $f_i$'s by $e_i$'s and so is its analogue about expressions of the form $P(e)\ell$.
\end{remark}

\begin{corollary} \label{cor:lastcor}
Consider any non-commutative monomial $Q(f)\in \uvsln$ in the $f_i$'s.
\begin{enumerate}
\item The element $Q(f)$ can be expressed a linear combination of monomials (having the same length as $Q(f)$) of the form:
\[f_1f_2\cdots f_s,\qquad f_1^2 Q'(f), \qquad f_t Q'(f)\]
for some $0\leq s\leq n-1$, $1<t\leq n-1$ and some non-commutative monomials $Q'(f)$.
\item The element $\ell Q(f)\in MU(n)$ can be expressed a linear combination of monomials (having the same length as $\ell Q(f)$) of the form:
\[Q'(f)\ell f_1f_2\cdots f_s\]
for some $0\leq s\leq n-1$ and some non-commutative monomials $Q'(f)$.
\end{enumerate}
\end{corollary}
\begin{proof}
Statement (2) follows directly from Lemma \ref{lem:std}, and (1) is a consequence of its proof.
\end{proof}

\subsection{Representations of $MU(n)$} \label{sec:mirareps}
\subsubsection{Weight Spaces}\label{sec:weightspaces}
Recall the theory of weight modules for $\uvsln$ (see for example \cite[\S 5]{J}). To be explicit, we define an equivalence relation on $\ZZ^n$, given by the orbits for the $\ZZ$-action
$$r\cdot (\lambda_1,\ldots,\lambda_n)=(\lambda_1+r,\ldots,\lambda_n+r),\qquad r,\lambda_1,\ldots,\lambda_n\in\ZZ.$$ Then a \emph{weight} will be an equivalence class $\lambda\in\ZZ^n/\ZZ$ and we call $\ZZ^n/\ZZ$ the \emph{weight lattice}. With a slight abuse, we will also refer to a representative $\lambda=(\lambda_1,\ldots,\lambda_n)\in\ZZ^n$ as a weight, when it does not create confusion, and we will often choose the representative with $\lambda_n=0$. Note that we only work with integral weights because we focus on finite dimensional representations. We define the \emph{fundamental weights} $\omega_j=(1,1,\cdots,1,0,0,\cdots,0)=(1^j0^{n-j})$ ($1$ occurs $j$ times) for $0\leq j\leq n-1$, which span the weight lattice $\ZZ^n/\ZZ$ as a $\ZZ$-module. 
With these conventions, the inner product on the weight lattice (notice that it is well defined on equivalence classes and in general it takes values in $\mathbb{Q}$) is given by
\[(\lambda,\mu)=\sum_{i=1}^n\lambda_i\mu_i-\frac{1}{n}\left(\sum_{i=1}^n\lambda_i\right)\left(\sum_{i=1}^n\mu_i\right).\]
We then define the \emph{simple roots} 
$\alpha_i:=(0,\ldots,0,\stackrel{i}{1},\stackrel{i+1}{-1},0,\ldots,0)$ for $1\leq i\leq n-1$, which are dual to the fundamental weights in the sense that
\[(\omega_j,\alpha_i)=\delta_{i,j} \quad\text{ for } 1\leq i,j\leq n-1.\]
The $\ZZ$-submodule of $\ZZ^n/\ZZ$ spanned the simple roots is called the \emph{root lattice}.
Every finite dimensional module $M$ for $\uvsln$ decomposes as a direct sum of weight spaces $M=\bigoplus_{\lambda,\sigma}M_{\lambda,\sigma}$, where $\lambda=(\lambda_1,\ldots,\lambda_{n})$ is a weight, $\sigma=(\sigma_1,\ldots,\sigma_{n-1})\in\{\pm 1\}^{n-1}$ is a choice of signs, and
$$M_{\lambda,\sigma}:=\{x\in M~|~k_i\cdot x=\sigma_i v^{\lambda_i-\lambda_{i+1}}x,\quad\forall i=1,\ldots,n-1\}.$$

If $M=M^\sigma$, where for each $\sigma\in\{\pm 1\}^{n-1}$, $M^\sigma:=\oplus_{\lambda}M_{\lambda,\sigma}$, then we say that $M$ is of type $\sigma$. The category of finite dimensional $\uvsln$-modules decomposes as a direct sum (over all $\sigma$) of the categories of modules of type $\sigma$, and the direct summands are all equivalent to the category of modules of type $\bold{1}=(1,\ldots,1)\in\{\pm 1\}^{n-1}$ (\cite[\S 5.2]{J}). It is therefore enough to restrict to such modules.

We say that a weight $\lambda$ is \emph{dominant} if $(\lambda,\alpha_i)\geq 0$ for all $i$. This is equivalent to $\lambda_1\geq\lambda_2\geq\cdots\geq\lambda_{n}$ (notice that this condition is well defined on equivalence classes), and also to the condition that $\lambda$ is a nonnegative integer linear combination of the fundamental weights.

We define a partial order on the set of weights by saying that $\lambda\geq \mu$ if and only if we can pick representatives for these weights in $\ZZ^n$ 
such that $\lambda-\mu$ is a non-negative integer linear combination of the simple roots $\alpha_i$. If this is the case, we say that $\lambda$ is \emph{higher} than $\mu$.

In the category of modules of type $\bold{1}$, for each dominant integral weight $\lambda$, 
there is a unique simple finite dimensional module $L_\lambda$ of highest weight $\lambda$ and all simple finite dimensional modules are of this form. 

We note at this point that the action of the Weyl group on the weights can just be identified with the action of the symmetric group $\mathcal{S}_n$ on $\ZZ^n$ by permutation of the factors, descended to the equivalence classes.

Now let $M$ be a finite dimensional representation for $MU(n)$, which in particular can be seen as a $\uvsln$-representation given the inclusion of Proposition \ref{prop:comodule-alg}(2). Then, it decomposes as a direct sum of weight spaces for the action of the $k_i$'s. Since $\ell$ is an idempotent, it is diagonalizable with eigenvalues $0$ and $1$. Furthermore, it commutes with all the $k_i$'s, and hence we can further refine the weight spaces to
$$M_{\lambda,\sigma}=M_{\lambda,\sigma,0}\oplus M_{\lambda,\sigma,1}$$
where, for $\epsilon=0,1$,
$$M_{\lambda,\sigma,\epsilon}:=\{x\in M~|~\ell \cdot x=\epsilon x,~k_i\cdot x=\sigma_i v^{\lambda_i-\lambda_{i+1}}x,\quad\forall i=1,\ldots,n-1\}.$$

\begin{definition}
We refer to $M_{\lambda,\sigma}$ as \emph{weight spaces}, and to an element $x\in M_{\lambda,\sigma}$ as a \emph{weight vector}. We will use the terms \emph{mirabolic weight spaces} and \emph{mirabolic weight vectors}, respectively, to refer to $M_{\lambda,\sigma,\epsilon}$ and its elements. So, in particular, a mirabolic weight vector is also a weight vector but the vice-versa is not necessarily true. A (mirabolic) \emph{highest weight vector} is a (mirabolic) weight vector $x$ such that $e_i\cdot x=0$ for all $i=1,\ldots,n-1$.\end{definition}

Exactly like in the case of $\uvsln$, we can define modules of type $\sigma$ for $MU(n)$ and, for the exact same reasons as in \cite[\S 5.2]{J}, we have that the category of finite dimensional modules for $MU(n)$ decomposes as a direct sum of the categories of modules of type $\sigma$, and the automorphism $\widetilde{\sigma}:MU(n)\to MU(n)$ given by
$$\widetilde{\sigma}(e_i)=\sigma_ie_i,\quad \widetilde{\sigma}(f_i)=f_i,\quad\widetilde{\sigma}(k_i^{\pm 1})=\sigma_ik_i^{\pm 1},\quad\widetilde{\sigma}(\ell)=\ell,$$
gives an equivalence of categories between modules of type $\bold{1}$ and modules of type $\sigma$. Therefore, since no information is lost and all the other modules can be obtained by twisting with the automorphism $\widetilde{\sigma}$, from now on we will assume that all modules are of type $\bold{1}$ and we will write (mirabolic) weight spaces more simply as
$$M_{\lambda}:=M_{\lambda,\bold{1}},\quad M_{\lambda,\epsilon}:=M_{\lambda,\bold{1},\epsilon}\qquad\forall \lambda,~\epsilon.$$
\subsubsection{Mirabolic Representations}\label{sec:mir-reps}
Let $V=K^n$ be the defining representation of $\uvsln$, with basis $\{u_1, u_2, \cdots, u_n\}$, 
where the action is given by
\[k_iu_j=v^{\delta_{i,j}-\delta_{i+1,j}}u_j;\qquad e_iu_j=\delta_{i+1,j}u_i;\qquad f_iu_j=\delta_{i,j}u_{i+1}.\]
\begin{definition}For any $0\leq r\leq n$, we consider the quantum wedge product $W_r:=\wedge_v^r(V)\subset V^{\otimes r}$, defined to be the span of all elements of the form
\[u_{i_1}\wedge\cdots\wedge u_{i_r}:=\sum_{\sigma\in \mathcal S_r}(-v^{-1})^{\len(\sigma)}u_{i_{\sigma(1)}}\otimes\cdots\otimes u_{i_{\sigma(r)}}\]
where $1\leq i_1<\cdots<i_r\leq n$, and $\len(\cdot)$ is the length function on the symmetric group $\mathcal S_r$.\end{definition}
\begin{definition}\label{def:lexi}
For any subset $I\subset\{1,\ldots,n\}$, we can order its elements in increasing order, say $I=\{i_1<\ldots<i_r\}$, and we will denote $u_I:=u_{i_1}\wedge\cdots\wedge u_{i_r}$. We also define the lexicographic total order on the set $[n]_r:=\{I\subset\{1,\ldots,n\}~|~\#I=r\}$ by saying that $\{i_1<\ldots<i_r\}< \{i'_1<\ldots< i'_r\}$ if there is an index $1\leq m\leq r$ such that $i_1=i'_1$, \ldots, $i_{m-1}=i'_{m-1}$ and $i_{m}<i'_m$.\end{definition}

Then $W_r:= \wedge_v^r(V)$ is in fact isomorphic to the simple highest weight representation $L_{\omega_r}$ of $\uvsln$ with highest weight $\omega_r=(1^r0^{n-r})$.

\begin{remark}
To write the action of $\uvsln$ on $W_r$, we consider the action on $V^{\otimes r}$, which is given by the iteration of the coproduct, so we get
\[\Delta^{r-1}(k_i)=k_i^{\otimes r};\quad \Delta^{r-1}(e_i)=\sum_{j=0}^{r-1} 1^{\otimes j}\otimes e_i\otimes k_i^{\otimes r-j};\quad \Delta^{r-1}(f_i)=\sum_{j=0}^{r-1} (k_i^{-1})^{\otimes j}\otimes f_i\otimes 1^{\otimes r-j}.\]
Applying that to \[u_I:=u_{i_1}\wedge\cdots\wedge u_{i_r}=\sum_{\sigma\in \mathcal S_r}(-v^{-1})^{\len(\sigma)}u_{i_{\sigma(1)}}\otimes\cdots\otimes u_{i_{\sigma(r)}}\] we get that,
\[ k_i\cdot u_I=\begin{cases}v u_I & \text{ if }i\in I,~ i+1\not\in I,\\
v^{-1} u_I & \text{ if }i\not\in I,~ i+1\in I,\\
u_I & \text{ otherwise,}\end{cases}\quad e_i\cdot u_I=\begin{cases}u_{I\setminus\{i+1\}\cup\{i\}} & \text{ if }i+1\in I,~ i\not\in I,\\
0 & \text{ otherwise,}\end{cases}\]
\[ f_i\cdot u_I=\begin{cases}u_{I\setminus\{i\}\cup\{i+1\}} & \text{ if }i\in I,~ i+1\not\in I,\\
0 & \text{ otherwise.}\end{cases}\]
\end{remark}

\begin{lemma}
We can define an $MU(n)$-action on $W_r$ as follows. For any subset $I\subset\{1,\ldots,n\}$,
define:
\[\ell\cdot u_I=\begin{cases} u_I & \text{ if }1\not\in I,\\
0 & \text{ if }1\in I.   
\end{cases}\]
\end{lemma}
\begin{proof}
We only need to check the relations in Definition \ref{def:mir-sln}.
\[\ell^2\cdot u_I=\begin{cases} \ell\cdot u_I & \text{ if }1\not\in I,\\
\ell\cdot 0 & \text{ if }1\in I \end{cases}=\begin{cases} u_I & \text{ if }1\not\in I, \\
 0 & \text{ if }1\in I\end{cases}=\ell\cdot u_I.\]
 If $c\in \ZZ$ is such that $k_iu_I=v^c u_I$, then for all $i=1,\ldots,n-1$,
 \[k_i\ell\cdot u_I=\begin{cases} k_i\cdot u_I & \text{ if }1\not\in I,\\
k_i\cdot 0 & \text{ if }1\in I \end{cases}=\begin{cases} v^c u_I & \text{ if }1\not\in I,\\ 0 & \text{ if }1\in I\end{cases}=\ell k_i\cdot u_I.\]
Let $i\geq 2$, then $1\in I$ if and only if $1\in I\setminus \{i\}\cup\{i+1\}$, hence
\begin{align*}\ell e_i\cdot u_I&=\begin{cases} \ell \cdot u_{I\setminus \{i\}\cup\{i+1\}} & \text{ if }i\not\in I,~i+1\in I,\\
\ell\cdot 0 & \text{ otherwise } \end{cases} \\
& =\begin{cases} u_{I\setminus \{i\}\cup\{i+1\}} & \text{ if }1,i\not\in I,~i+1\in I, \\ 0 & \text{ otherwise }\end{cases} \\ 
&=e_i\ell\cdot u_I.\end{align*}
Similarly, $1\in I$ if and only if $1\in I\setminus \{i+1\}\cup\{i\}$, hence
\begin{align*}\ell f_i\cdot u_I&=\begin{cases} \ell\cdot u_{I\setminus \{i+1\}\cup\{i\}} & \text{ if }i+1\not\in I,~i\in I,\\
\ell\cdot 0 & \text{ otherwise } \end{cases} \\
& =\begin{cases} u_{I\setminus \{i+1\}\cup\{i\}} & \text{ if }1,i+1\not\in I,~i\in I,\\ 0 & \text{ otherwise }\end{cases} \\ 
&=f_i\ell\cdot u_I.\end{align*}
Then,
\begin{align*}\ell e_1\ell\cdot u_I&=\begin{cases} \ell e_1\cdot u_{I} & \text{ if }1\not\in I, \\
\ell e_1\cdot 0 & \text{ if }1\in I \end{cases} \\
&=\begin{cases} \ell e_1\cdot u_{I} & \text{ if }1\not\in I,\\
 0 & \text{ if }1\in I \end{cases} \\
 &=\begin{cases} \ell e_1\cdot u_{I} & \text{ if }1\not\in I,\\
 \ell e_1\cdot u_{I} & \text{ if }1\in I \end{cases} \\
&=\ell e_1\cdot u_I.\end{align*}
Also,
\begin{align*}f_1\ell\cdot u_I&=\begin{cases} f_1\cdot u_{I} & \text{ if }1\not\in I, \\
f_1\cdot 0 & \text{ if }1\in I \end{cases} \\
&=\begin{cases} 0 & \text{ if }1\not\in I, \\
 0 & \text{ if }1\in I \end{cases} \\
 &=0 \\
 &=\ell f_1\ell\cdot u_I.\end{align*}
We then compute 
\begin{align*}v^{-1}e_1^2\ell\cdot u_I&=\begin{cases}v^{-1}e_1^2\cdot u_{I} & \text{ if }1\not\in I, \\
 0 & \text{ if }1\in I \end{cases} \\
&=\begin{cases} 0 & \text{ if }1\not\in I,~2\not\in I, \\ v^{-1}e_1\cdot u_{I\setminus\{2\}\cup\{1\}} & \text{ if }1\not\in I,~2\in I,\\
 0 & \text{ if }1\in I \end{cases} \\
 &=\begin{cases} 0 & \text{ if }1\not\in I,~2\not\in I, \\ 0 & \text{ if }1\not\in I,~2\in I,\\
 0 & \text{ if }1\in I \end{cases} \\
 &=0,\end{align*}
\begin{align*}v\ell e_1^2\cdot u_I&=\begin{cases}v\ell e_1\cdot u_{I\setminus\{2\}\cup\{1\}} & \text{ if }1\not\in I,~2\in I,\\
 0 & \text{ otherwise }\end{cases} \\
 &=\begin{cases}0 & \text{ if }1\not\in I,~2\in I,\\
 0 & \text{ otherwise }\end{cases} \\
&=0.\end{align*}
Hence
\begin{align*}(v+v^{-1})e_1\ell e_1\cdot u_I&=\begin{cases}(v+v^{-1})e_1\ell\cdot u_{I\setminus\{2\}\cup\{1\}} & \text{ if }1\not\in I,~2\in I,\\
 0 & \text{ otherwise }\end{cases} \\
  &=\begin{cases}0 & \text{ if }1\not\in I,~2\in I,\\
 0 & \text{ otherwise }\end{cases} \\
 &=0 \\
 &= (v\ell e_1^2+v^{-1}e_1^2\ell)\cdot u_I.\end{align*}
We then compute 
\begin{align*}vf_1^2\ell\cdot u_I&=\begin{cases}vf_1^2\cdot u_{I} & \text{ if }1\not\in I, \\
 0 & \text{ if }1\in I \end{cases} \\
 &=\begin{cases}0 & \text{ if }1\not\in I, \\
 0 & \text{ if }1\in I \end{cases} \\
 &=0,
\end{align*}
\begin{align*}v^{-1}\ell f_1^2\cdot u_I&=\begin{cases}v^{-1}\ell f_1\cdot u_{I\setminus\{1\}\cup\{2\}} & \text{ if }1\in I,~2\not\in I, \\
 0 & \text{ otherwise }\end{cases} \\
 &=\begin{cases}0 & \text{ if }1\in I,~2\not\in I,\\
 0 & \text{ otherwise }\end{cases} \\
&=0.\end{align*}
Hence
\begin{align*}(v+v^{-1})f_1\ell f_1\cdot u_I&=\begin{cases}(v+v^{-1})f_1\ell\cdot u_{I\setminus\{1\}\cup\{2\}} & \text{ if }1\in I,~2\not\in I,\\
 0 & \text{ otherwise }\end{cases} \\
 &=\begin{cases}(v+v^{-1})f_1\cdot u_{I\setminus\{1\}\cup\{2\}} & \text{ if }1\in I,~2\not\in I,\\
 0 & \text{ otherwise }\end{cases} \\
&=\begin{cases}0 & \text{ if }1\in I,~2\not\in I,\\
 0 & \text{ otherwise }\end{cases} \\
 &=0 \\
 &= (vf_1^2\ell+v^{-1}\ell f_1^2)\cdot u_I.\end{align*}
\end{proof}
\begin{remark}
Notice that $W_0$ and $W_n$ are both the trivial one dimensional representations of weight $(0,\ldots,0)=(1,\ldots,1)$ for $\uvsln$, with the difference that $\ell$ acts on $W_0$ by the identity and on $W_n$ by $0$.     
\end{remark}

\begin{remark}
The action of $MU(n)$ on the representation $W_r$ can be viewed as being given by `quantum difference operators'. (See \cite[5A.6]{J} for more details regarding $\uvsln$-representations.) On the polynomial ring $K[y_1, y_2,\cdots, y_k]$, we define the operators $\mm_{y_t}^{\pm1}$ and $\dd_{y_t}$ for $1\leq t\leq k$ by:
\[(\mm_{y_t}^{\pm1}\cdot f)(y_1,y_2,\dots, y_k): = f(y_1,y_2,\cdots, v^{\pm1}y_t,\cdots, y_k),\]
\[\dd_{y_t}\cdot f:=\frac{\mm_{y_t}\cdot f - \mm_{y_t}^{-1}\cdot f }{(v-v^{-1})y_t}.
\]
For any $1\leq r\leq n$, we construct the $n\times r$ matrix of indeterminates $X:=(x_{i,j})$. Consider the polynomial ring $\CC[x_{i,j}:1\leq i\leq n, 1\leq j\leq r]$ on which we have the operators $\dd_{i,j}:=\dd_{x_{i,j}}$ and $\mm_{i,j}^{\pm1}:=\mm_{x_{i,j}}^{\pm1}$. The $MU(n)$-representation $W_r$ can be viewed as the subspace of this ring spanned by all $r\times  r$ minors of the matrix $X$, on which the action of $MU(n)$ is given by the following operators:
\begin{align*}
e_i \mapsto \sum_{t=1}^r x_{i,t} \dd_{i+1,t},& \qquad \qquad f_i \mapsto \sum_{t=1}^r x_{i+1,t} \dd_{i,t},\\
k_i^{\pm 1} \mapsto  \sum_{t=1}^r \mm_{i,t}^{\pm1}\mm_{i+1,t}^{\mp1},& \qquad \qquad 1-\ell\mapsto \sum_{t=1}^r x_{1,t} \dd_{1,t}.
\end{align*}
We note that the action of $1-\ell$ is given by an `Euler-type' difference operator that degenerates to the Euler vector field in the top-row variables as $v\to 1$.
\end{remark}

\begin{definition}
Let $\lambda$ be a dominant integral weight, $L_{\lambda}$ be the simple representation of $\uvsln$ with highest weight $\lambda$, and $0\leq r\leq n$. We define $L_{\lambda, r}:=L_\lambda\otimes W_r$, and we note that this has the structure of an $MU(n)$-representation via the comodule map $\rho$.
\end{definition}

Let $x_{\lambda}$ be a highest weight vector of the $\uvsln$-representation $L_{\lambda}$ and let $x_{\lambda,r}: = x_{\lambda} \otimes u\in L_{\lambda,r}$, where $u=u_{\{1,\ldots,r\}}=u_1\wedge u_2\wedge \cdots \wedge u_r \in W_r$. It is clear that $x_{\lambda,r}$ is a mirabolic highest weight vector of the $MU(n)$-representation $L_{\lambda,r}$.

\begin{example}
Consider the dominant integral weight $\lambda=(2,1,0)$ for $U_v(\mathfrak{sl}_3)$. Here, we diagrammatically represent the mirabolic weight space decompositions of the $MU(3)$-representations $L_{\lambda,0}$ and $L_{\lambda,1}$.

In the diagram below, each circular dot represents a one-dimensional weight space whereas the square represents a two-dimensional space. The solid lines indicate the action of $f_1$ and the dashed lines indicate the action of $f_2$. We have labeled the vertices with the $(k_1,k_2;\ell)$-weights. (The central vertex represents the zero weight space.)

\begin{center}
\begin{tikzpicture}[xscale=0.4,yscale=0.4]
\draw(-22,14) node {$L_{\lambda,0}:$};

\draw(-13,18.7) node {$(v,v;0)$};
\filldraw (-12,18) circle(5pt);

\draw(-7,18.7) node {$(v^{-1},v^2;0)$};
\filldraw (-8,18) circle(5pt);

\draw(-16.5,14) node {$(v^2,v^{-1};0)$};
\filldraw (-14,14) circle(5pt);

\filldraw[very thick] (-10.15,13.85) rectangle (-9.85,14.15);

\draw(-3.7,14) node {$(v^{-2},v;0)$};
\filldraw (-6,14) circle(5pt);

\draw(-13,9.2) node {$(v,v^{-2};0)$};
\filldraw (-12,10) circle(5pt);

\draw(-7,9.2) node {$(v^{-1},v^{-1};0)$};
\filldraw (-8,10) circle(5pt);

\draw[mid arrow] (-12,18) to (-8,18) [thin ];
\draw[mid arrow] (-14,14) to (-10,14) [thin ];
\draw[mid arrow] (-10,14) to (-6,14) [thin ];
\draw[mid arrow] (-12,10) to (-8,10) [thin ];

\draw[mid arrow] (-12,18) to (-14,14) [dashed];
\draw[mid arrow] (-8,18) to (-10,14) [dashed];
\draw[mid arrow] (-10,14) to (-12,10) [dashed];
\draw[mid arrow] (-6,14) to (-8,10) [dashed];
\end{tikzpicture}
\end{center}

The following diagram follows the same conventions as the one above, except we have only indicated the highest and lowest weights. The three layers represent the subspaces $L_{\lambda}\otimes\{u_1\}, L_{\lambda}\otimes\{u_2\}$ and $L_{\lambda}\otimes\{u_3\}$ of $L_{\lambda, 1}$ respectively from top to bottom.

\begin{center}
\begin{tikzpicture}[xscale=0.5,yscale=0.4]
\draw(-18,6) node {$L_{\lambda,1}:$};

\draw(-13,16.8) node {$(v^2,v;0)$};
\filldraw (-12,16) circle(5pt);
\filldraw (-8,16) circle(5pt);
\filldraw (-14,14) circle(5pt);
\filldraw[very thick] (-10.15,13.85) rectangle (-9.85,14.15);
\filldraw (-6,14) circle(5pt);
\filldraw (-12,12) circle(5pt);
\filldraw (-8,12) circle(5pt);

\draw[mid arrow] (-12,16) to (-8,16) [thin ];
\draw[mid arrow] (-14,14) to (-10,14) [thin ];
\draw[mid arrow] (-10,14) to (-6,14) [thin ];
\draw[mid arrow] (-12,12) to (-8,12) [thin ];

\draw[mid arrow] (-12,16) to (-14,14) [dashed];
\draw[mid arrow] (-8,16) to (-10,14) [dashed];
\draw[mid arrow] (-10,14) to (-12,12) [dashed];
\draw[mid arrow] (-6,14) to (-8,12) [dashed];

\draw[mid arrow] (-12,16) to (-8,8) [thin ];
\draw[mid arrow] (-8,16) to (-4,8) [thin ];
\draw[mid arrow] (-14,14) to (-10,6) [thin ];
\draw[mid arrow] (-10,14) to (-6,6) [thin ];
\draw[mid arrow] (-6,14) to (-2,6) [thin ];
\draw[mid arrow] (-12,12) to (-8,4) [thin ];
\draw[mid arrow] (-8,12) to (-4,4) [thin ];

\filldraw (-8,8) circle(5pt);
\filldraw (-4,8) circle(5pt);
\filldraw (-10,6) circle(5pt);
\filldraw[very thick] (-6.15,5.85) rectangle (-5.85,6.15);
\filldraw (-2,6) circle(5pt);
\filldraw (-8,4) circle(5pt);
\filldraw (-4,4) circle(5pt);

\draw[mid arrow] (-8,8) to (-4,8) [thin ];
\draw[mid arrow] (-10,6) to (-6,6) [thin ];
\draw[mid arrow] (-6,6) to (-2,6) [thin ];
\draw[mid arrow] (-8,4) to (-4,4) [thin ];

\draw[mid arrow] (-8,8) to (-10,6) [dashed];
\draw[mid arrow] (-4,8) to (-6,6) [dashed];
\draw[mid arrow] (-6,6) to (-8,4) [dashed];
\draw[mid arrow] (-2,6) to (-4,4) [dashed];

\draw[mid arrow] (-8,8) to (-10,0) [dashed];
\draw[mid arrow] (-4,8) to (-6,0) [dashed];
\draw[mid arrow] (-10,6) to (-12,-2) [dashed];
\draw[mid arrow] (-6,6) to (-8,-2) [dashed];
\draw[mid arrow] (-2,6) to (-4,-2) [dashed];
\draw[mid arrow] (-8,4) to (-10,-4) [dashed];
\draw[mid arrow] (-4,4) to (-6,-4) [dashed];

\filldraw (-10,0) circle(5pt);
\filldraw (-6,0) circle(5pt);
\filldraw (-12,-2) circle(5pt);
\filldraw[very thick] (-8.15,-2.15) rectangle (-7.85,-1.85);
\filldraw (-8,-2) circle(5pt);
\filldraw (-4,-2) circle(5pt);
\filldraw (-10,-4) circle(5pt);
\draw(-5,-4.8) node {$(v^{-1},v^{-2};1)$};
\filldraw (-6,-4) circle(5pt);

\draw[mid arrow] (-10,0) to (-6,0) [thin ];
\draw[mid arrow] (-12,-2) to (-8,-2) [thin ];
\draw[mid arrow] (-8,-2) to (-4,-2) [thin ];
\draw[mid arrow] (-10,-4) to (-6,-4) [thin ];

\draw[mid arrow] (-10,0) to (-12,-2) [dashed];
\draw[mid arrow] (-6,0) to (-8,-2) [dashed];
\draw[mid arrow] (-8,-2) to (-10,-4) [dashed];
\draw[mid arrow] (-4,-2) to (-6,-4) [dashed];
\end{tikzpicture}
\end{center}

(Similar diagrams for $MU(2)$-representations were given in \cite[Example 5.9]{R15}.)
\end{example}

\begin{lemma} \label{lem:highcyclic}
The space $L_{\lambda,r}$ is generated by the vector $x_{\lambda,r}$ as an $MU(n)$-module.
\end{lemma}
\begin{proof}
When $r=0$ or $r=n$, the result is clear since $L_{\lambda,0}=\pi_1^*(L_{\lambda})$ and $L_{\lambda,n}=\pi_0^*(L_{\lambda})$ are the pullbacks along the maps defined in Prop. \ref{prop:comodule-alg}(1), and $L_{\lambda}$ is a simple $\uvsln$-module. Therefore, we assume that $1\leq r\leq n-1$. 

Note that the space $L_{\lambda}$ is spanned by elements of the form $Q(f)\cdot x_{\lambda}$, where $Q(f)$ is a non-commutative monomial in the $f_i$'s.
We start by proving the claim that for all non-commutative monomials $Q(f)$ and $r$-element subsets $I\sub \{1,2,\dots, r+1\}$, we have
\[(Q(f)\cdot x_{\lambda})\otimes u_I\in MU(n)\cdot x_{\lambda,r}.\]
We proceed by induction on the length of $Q(f)$ in terms of the $f_i$'s.
 For the base case, suppose $\deg(Q(f))=0$. We need to show that $x_{\lambda}\otimes u_I \in MU(n)\cdot x_{\lambda,r}$. To this end, we compute:
\begin{align*}\ell f_1f_2 \cdots f_r\cdot x_{\lambda,r} &= \ell f_1f_2 \cdots f_r \cdot (x_{\lambda}\otimes u) \\
&= (1\otimes \ell)\left(k_1^{-1}\cdots k_r^{-1}x_\lambda\otimes u_{\{2,\ldots,r+1\}}+\sum_{J\in[n]_r,~1\in J}x_J\otimes u_J\right)\\
&= v^{\sum_{j=1}^r(\lambda_{j+1}-\lambda_j)}x_{\lambda}\otimes u_{\{2,\ldots,r+1\}}.\end{align*}
Hence $x_{\lambda}\otimes u_{\{2,\ldots,r+1\}}\in MU(n)\cdot x_{\lambda,r}$. Next, for any $1\leq i\leq r$, we see that:
\[e_i e_{i-1}\cdots e_1 \cdot (x_{\lambda}\otimes u_{\{2,\ldots,r+1\}}) = x_{\lambda}\otimes u_{I_{i+1}},\]
where $I_i =\{1,2,\dots, r+1\}\sm\{i\}$. This proves the base case. Now suppose that the claim is true for a given $Q(f)$, so in particular we have \[(Q(f)\cdot x_{\lambda})\otimes  u,~(Q(f)\cdot x_{\lambda})\otimes  u_{\{2,\ldots,r+1\}} \in MU(n)\cdot x_{\lambda,r}.\]
Then for all $j\neq r+1$ (this condition is vacuous when $r=n-1$), we have that:
\[f_j\cdot ((Q(f)\cdot x_{\lambda})\otimes u_{\{2,\ldots,r+1\}}) = (f_j Q(f)\cdot x_{\lambda})\otimes u_{\{2,\ldots,r+1\}}\in MU(n)\cdot x_{\lambda,r}\]
since $f_j\cdot u_{\{2,\ldots,r+1\}}=0$.
And we also have that:
\[f_{r+1}\cdot ((Q(f)\cdot x_{\lambda})\otimes u) = (f_{r+1} Q(f)\cdot x_{\lambda})\otimes u\in MU(n)\cdot x_{\lambda,r},\]
since $f_{r+1}\cdot u=0$.
Then, just as in the base case, we observe that:
\[\ell f_1f_2 \cdots f_r\cdot ((f_{r+1} Q(f)\cdot x_{\lambda})\otimes u) = d(f_{r+1} Q(f)\cdot x_{\lambda})\otimes u_{\{2,\ldots,r+1\}},\]
where $d$ is the constant from the equation $d(f_{r+1} Q(f)\cdot x_{\lambda}) = k_1^{-1}\cdots k_r^{-1}(f_{r+1} Q(f)\cdot x_{\lambda})$.

This shows that $(\widetilde Q(f)\cdot x_{\lambda})\otimes u_{\{2,\ldots,r+1\}} \in MU(n)\cdot x_{\lambda,r}$ whenever $\deg\widetilde Q(f)\leq\deg Q(f)+1$.
Finally, we compute for $1\leq i\leq r-1$:
\begin{align}
&e_i e_{i-1}\cdots e_1 \cdot ((\widetilde Q(f)\cdot x_{\lambda})\otimes u_{\{2,\ldots,r+1\}})\\
&= \sum_{j=0}^{i}(e_i e_{i-1}\cdots e_{j+1}\widetilde Q(f)\cdot x_{\lambda})\otimes (k_i\cdots k_{j+1} e_{j}\cdots e_1\cdot u_{\{2,\ldots,r+1\}}). \label{eq:bigsum1}
\end{align}
Since $e_k\cdot x_{\lambda}=0$ for all $k$, we have:
\begin{align*}
e_ie_{i-1}\cdots e_{j+1}\widetilde Q(f)\cdot x_{\lambda}
&=e_ie_{i-1}\cdots e_{j+2}[e_{j+1},\widetilde Q(f)]\cdot x_{\lambda}\\
&=\vdots\\
&=[e_i,[e_{i-1},[\cdots,[e_{j+1},Q'(f)]\cdots]]]\cdot x_{\lambda}.
\end{align*}
The iterated commutator in the last equality above can be simplified using the relation~\eqref{eq:usln.3}  to get that it is equal to a linear combination of monomials in the $f_t$'s each having length lesser than or equal to $\deg Q(f)$ (up to a factor of some $k_t$'s). Then, by the induction assumption, we have that the summands in \eqref{eq:bigsum1} for $j<i$ all lie in $MU(n)\cdot x_{\lambda,r}$.
This implies that the same must be true for the summand when $j=i$, which is equal to:
\[(\widetilde Q(f)\cdot x_{\lambda})\otimes (e_{i} e_{i-1}\cdots e_1\cdot u_{\{2,\ldots,r+1\}})=(\widetilde Q(f)\cdot x_{\lambda})\otimes u_{I_i},\]
where $I_i=\{1,2,\dots, r+1\}\sm\{i\}$. This completes the proof of the claim.

To conclude the proof of the Lemma, we use induction on the totally ordered set $[n]_r$ with the lexicographic order, as defined in Def. \ref{def:lexi}, to show that $L_\lambda\otimes \{u_I\}\subset MU(n)\cdot x_{\lambda,r}$ for all $I\in[n]_r$. The claim that we have proved shows that $L_\lambda\otimes \{u\}\subset MU(n)\cdot x_{\lambda,r}$ which is our base case. Now suppose that $I=\{i_1,\ldots,i_r\}\neq\{1,\ldots,r\}$, we want to prove that for any weight vector $x\in L_\lambda$, $x\otimes u_I\in MU(n)\cdot x_{\lambda,r}$. Since $I\neq\{1,\ldots,r\}$, there exists an index $1\leq j\leq r$ such that $i_j-1\in [n]\setminus I$ and therefore $I'_j:=I\setminus\{i_j\}\cup\{i_j-1\}\in[n]_r$. By induction hypothesis, since $I'_j<I$ in the lexicographic order, $x\otimes u_{I'_j}$, $f_{i_j-1}x\otimes u_{I'_j}\in MU(n)\cdot x_{\lambda,r}$. We compute
\begin{align*}f_{i_j-1}\cdot\left(x\otimes u_{I'_j}\right)&=k_{i_j-1}^{-1}x\otimes f_{i_j-1}u_{I'_j}+f_{i_j-1}x\otimes u_{I'_j}\\
&=cx\otimes u_I+f_{i_j-1}x\otimes u_{I'_j}\end{align*}
where $0\neq c$ equals a power of $v$ because $x$ is a weight vector. Since all the other terms are in $MU(n)\cdot x_{\lambda,r}$, then so is $cx\otimes u_I$, which concludes the proof.
\end{proof}

\begin{proposition} \label{prop:irred}
Given a dominant integral weight $\lambda$, and $r$ satisfying $0\leq r\leq n$, the $MU(n)$-representation $L_{\lambda, r}$ is irreducible.
\end{proposition}

\begin{proof}
Again, when $r=0$ or $r=n$, the claim is clear since $L_{\lambda,0}=\pi_1^*(L_{\lambda})$ and $L_{\lambda,n}=\pi_0^*(L_{\lambda})$, and $L_{\lambda}$ is a simple $\uvsln$-module. Therefore, we assume that $1\leq r\leq n-1$.

Let $0\neq M\subset L_{\lambda,r}$ be an $MU(n)$-submodule, and let $0\neq x\in M$ be a highest weight vector. Our goal will be to prove that $x_{\lambda,r}\in M$, which will show that $M=L_{\lambda,r}$ by the previous lemma. 

Using the total order on $[n]_r$, we can write:
\[x=x_{I}\otimes u_I + \sum_{J\in [n]_r,~ J>I} x_J\otimes u_J,\]
where $I\in [n]_r$, $x_I,x_J\in L_\lambda$, $x_I\neq 0$. We call $x_{I}\otimes u_I$ the leading term of $x$. Note that if $I\not=\{1,2,\dots r\}$, then there exists some $j$ such that $j\not\in I$ and $j+1\in I$. In that case, $e_j$ acts non-trivially on $u_I$ and we have that $(I\sm \{j+1\})\cup{j}<I$. For such a $j$, this implies that $x_I\otimes (e_j\cdot u_I)$ is the leading term of $e_j\cdot x$, which in particular shows that $e_j\cdot x\neq 0$. However, since $x$ was assumed to be a highest weight vector, such a $j$ can not exist, implying that $I=\{1,2,\dots r\}$ and that the sum is over all $J\neq I$.

Next, we inductively construct a sequence of non-zero vectors $x_j$ in $MU(n)\cdot x$ for $1\leq j\leq r$ where:
\[x_j=\sum_{\substack{\{1,2,\dots, j\}\sub J}} x_{j,J}\otimes u_J,\]
for some $x_{j,J}\in L_{\lambda}$, such that $e_t\cdot x_j=0$ for all $t$, that is, $x_j$ is a highest weight vector. For the base case of the induction, since 
\[\ell\cdot(x_J\otimes u_J)=x_J\otimes \ell \cdot u_J=\begin{cases}0 & \text{ if }1\in J,\\ x_J\otimes u_J & \text{ if }1\not\in J,\end{cases}\]
we have
\[(1-\ell)\cdot x=x_I\otimes u_I+\sum_{J>I, ~1\in J}x_J\otimes u_J\] 
and by successively acting with the $e_t$'s we get to a vector $x_1\neq 0$ on which all the $e_t$'s act by zero. Notice that the condition $1\in J$ still holds for all the summands.

Next, suppose $x_{j}$ has been constructed for some $1\leq j\leq r-1$. By assumption, we have that $e_t\cdot x_j=0$ for all $t$. This implies that $e_t\cdot x_{j,J}=0$ for all $1\leq t\leq j$ and all $J$. Furthermore, since we also have $e_t\cdot x_j=0$ for $j+1\leq t\leq n-1$, there must be some $J$ such that $j+1\in J$ and $x_{j,J}\neq 0$.

We compute $\ell f_1\cdots f_j \cdot x_{j}$. Note that if $\{1,2,\dots, j+1\}\sub J$, it is clear that $\ell f_1\cdots f_{j} \cdot (w\otimes u_{J})=0$ for any $w$.
Therefore,
\begin{align*}
\ell f_1\cdots f_j \cdot x_j&=\sum_{\substack{\{1,2,\dots, j\}\sub J}} \ell f_1\cdots f_j \cdot(x_{j,J}\otimes u_J)\\
&=\sum_{\substack{\{1,2,\dots, j\}\sub J\\ j+1\not\in J}} \ell f_1\cdots f_j \cdot(x_{j,J}\otimes u_J)\\
&=\sum_{\substack{\{1,2,\dots, j\}\sub J\\ j+1\not\in J}}  (k_1^{-1}\cdots k_{j}^{-1}\cdot x_{j,J})\otimes (f_1\cdots f_j\cdot u_J)\\
&=\sum_{\substack{\{1,2,\dots, j\}\sub J\\ j+1\not\in J}}  c_{j,J} x_{j,J}\otimes (f_1\cdots f_j\cdot u_J),
\end{align*}
where the constants $c_{j,J}$ are determined by $c_{j,J}x_{j,J}=k_1^{-1}\cdots k_{j}^{-1}\cdot x_{j,J}$. Then, it follows that:
\begin{align*}
e_j\cdots e_1 \ell f_1\cdots f_j \cdot x_j&=\sum_{\substack{\{1,2,\dots, j\}\sub J\\ j+1\not\in J}}  c_{j,J} e_j\cdots e_1\cdot( x_{j,J}\otimes (f_1\cdots f_j\cdot u_J))\\
&=\sum_{\substack{\{1,2,\dots, j\}\sub J\\ j+1\not\in J}}  c_{j,J} x_{j,J}\otimes (e_j\cdots e_1 f_1\cdots f_j\cdot u_J))\\
&=\sum_{\substack{\{1,2,\dots, j\}\sub J\\ j+1\not\in J}}  c_{j,J} x_{j,J}\otimes u_J.
\end{align*}
Finally, we note that for all $J$'s above,
\[c_{j,J} x_{j,J}\otimes u_J = (k_1^{-1}\cdots k_{j}^{-1}\cdot x_{j,J})\otimes u_J = v(k_1^{-1}\cdots k_{j}^{-1})\cdot (x_{j,J}\otimes u_J)\]
Therefore, we can define:
\begin{align*}
\widetilde{x_{j+1}}:&=x_j - v^{-1} k_j\cdots k_{1} e_j\cdots e_1\ell f_1\cdots f_j \cdot x_j\\
&=\sum_{\substack{\{1,2,\dots, j+1\}\sub J}} x_{j,J}\otimes u_J \in MU(n)\cdot x.
\end{align*}
Note that we had concluded that there must be some $J$ such that $x_{j,J}\neq 0$ and $j+1\in J$. Therefore, we must have $\widetilde{x_{j+1}}\neq 0$. Finally, we can successively apply $e_t$'s on it till we get a vector $x_{j+1}$ on which each of the $e_t$'s act by zero. This completes the induction step.

Note that the element $x_r$ constructed in the sequence is exactly equal to $x_{r,I}\otimes u$ for some $x_{r,I}\in L_{\lambda}$. Furthermore, since it is a highest weight vector, it must be equal to (a non-zero scalar multiple of) $x_{\lambda,r}$ as the $\uvsln$-representation $L_{\lambda}$ has a unique highest weight vector (up to scalar multiples). This completes the proof of the lemma.
\end{proof}
\begin{proposition} \label{prop:nonisom}
The representations $L_{\lambda,r}$ are pairwise non-isomorphic.
\end{proposition}
\begin{proof}
To prove this, it suffices to show that if $(\lambda, r)\neq (\mu,s)$, then the representations $L_{\lambda,r}$ and $L_{\mu,s}$ do not have the same mirabolic weight space decomposition. We note that the $0$-eigenspace of $\ell$ in $L_{\lambda,r}$ has dimension $\dim(L_{\lambda}){n-1 \choose r-1}$ and the $1$-eigenspace has dimension $\dim(L_{\lambda}){n-1 \choose r}$. The ratio of these two dimensions is $[{n-1 \choose r-1}:{n-1 \choose r}]=[r:n-r]$. This uniquely determines the value of $r$, and so, we conclude that if $L_{\lambda,r}$ and $L_{\mu,s}$ have the same mirabolic weight space decomposition, then $r=s$. Next, since $\lambda+\omega_r$ is a weight of $L_{\lambda,r}$, it must also be a weight of $L_{\mu,r}$, whose weights are upper bounded by $\mu+\omega_r$, showing that $\lambda+\omega_r\leq \mu+\omega_r$, and hence $\lambda\leq\mu$. Similarly, we must also have $\mu\leq \lambda$, which implies that $\lambda=\mu$, completing the proof.
\end{proof}

\section{Depth of a representation}\label{sec:proof-prop}

In this section, we define the notion of `depth' of an $MU(n)$-representation and prove some properties that will be useful in the proof of the classification theorems.

\begin{definition}\label{def:type-r}
Let $M$ be a module for $MU(n)$, and let $x\in M$ be a weight vector. For $0\leq r\leq n$, we say that the $x$ has \emph{depth} $r$ if 
$$\ell f_1\cdots f_j \cdot x= 0\qquad \text{ for all }0\leq j<r,$$ 
and
$$\ell f_1f_2\cdots f_r\cdot  x\neq 0 \quad \text{ (this condition doesn't apply when $r=n$).}$$

If $M\neq 0$, we define the \emph{depth} of $M$ to be the maximum of the depth of $x$, where $x$ varies among all nonzero weight vectors in $M$.
\end{definition}
\begin{remark}
The $0$ vector in any module has depth $n$.
\end{remark}

\begin{example} \label{ex:depth}
Let $\lambda$ be a dominant integral weight and $0\leq r\leq n$. Pick any non-zero weight vector $x\in L_{\lambda}$ and $I\in [n]_r$. Then, the vector $x\otimes u_I$ is a weight vector of $L_{\lambda,r}$ whose depth is $s$, where $s$ is chosen to be the maximum number such that $\{1,2,\dots, s\}\sub I$. In particular, the depth of the module $L_{\lambda,r}$ is precisely $r$, which is what motivates the above definition.
\end{example}

\begin{lemma}
Let $M\neq 0$ be a module of $MU(n)$, then the depth of $M$ is also equal to the maximum of the depth of $x$, where $x$ varies among all nonzero mirabolic weight vectors in $M$. 
\end{lemma}
\begin{proof}
All mirabolic weight vectors are also weight vectors, so all that is needed here is to show that for all $0\leq r\leq n$, if there is a weight vector of depth $r$, then there is also a mirabolic weight vector of depth $r$. First suppose that $x$ is a weight vector of depth $0$, which means that $\ell \cdot x\neq 0$. Then $y=\ell\cdot x$ is a mirabolic weight vector such that $\ell\cdot y=\ell\cdot(\ell\cdot x)=\ell\cdot x=y$ so $y$ has depth $0$. Finally, if $x$ is a weight vector of depth $r$, with $r\geq 1$, then $\ell\cdot x=0$, so $x$ is a mirabolic weight vector.
\end{proof}

\begin{lemma}\label{lem:highest-depth} Let $M\neq 0$ be a finite dimensional module of $MU(n)$, then the depth of $M$ is also equal to the maximum of the depth of $x$, where $x$ varies among all nonzero highest weight vectors in $M$. 
\end{lemma}
\begin{proof}
To prove the lemma, we show that if $y$ is a weight vector of depth $r$, then for all $i=1,\ldots,n-1$, the vector $e_i\cdot y$ is a weight vector of depth $\geq r$. This gives the result because if there is a vector $y$ of maximum depth that is not a highest weight vector, then there is an index $i$ such that $e_i\cdot y\neq 0$, and $e_i\cdot y$ also has maximum depth. By iterating this procedure we eventually get a highest weight vector of maximum depth.

Suppose that $\ell f_1\cdots f_k \cdot y=0$, for all $0\leq k<r$, then we show that we also have $\ell f_1\cdots f_k e_i\cdot y=0$ for all $0\leq k<r$.
If $r=0$ the statement is vacuously true, so suppose that $r\geq 1$. For $k=0$, we have that $\ell\cdot y=0$ and $\ell e_iy=\ell e_i\ell\cdot y=\ell e_i\cdot(0)=0$.
Suppose that $k\geq 1$, and consider $\ell f_1\cdots f_k e_i\cdot y$. If $k+1\leq i\leq n-1$, then $\ell f_1\cdots f_k e_i\cdot y=e_i \ell f_1\cdots f_k\cdot y=0$. If $1\leq i\leq k$, then
\begin{align*}
\ell f_1\cdots f_k e_i\cdot y &= \ell f_1\cdots f_{i-1} f_i e_i f_{i+1}\cdots f_k\cdot y \\
&=  \ell f_1\cdots f_{i-1}\left( e_if_i -\frac{k_i-k_i^{-1}}{v-v^{-1}}\right)f_{i+1}\cdots f_k\cdot y \\
&=  \ell f_1\cdots f_{i-1}e_if_if_{i+1}\cdots f_k\cdot y -\ell f_1\cdots f_{i-1}\frac{k_i-k_i^{-1}}{v-v^{-1}}f_{i+1}\cdots f_k\cdot  y \\
&= \ell e_if_1\cdots f_{i-1}f_if_{i+1}\cdots f_k \cdot y - \frac{v^{-1}k_i-vk_i^{-1}}{v-v^{-1}}\ell f_1\cdots f_{i-1}f_{i+1}\cdots f_k\cdot y \\
&= \ell e_i\ell f_1\cdots f_{i-1}f_if_{i+1}\cdots f_k\cdot y - \frac{v^{-1}k_i-vk_i^{-1}}{v-v^{-1}}f_{i+1}\cdots f_k\ell f_1\cdots f_{i-1} \cdot y \\
&= \ell e_i\cdot(0) - \frac{v^{-1}k_i-vk_i^{-1}}{v-v^{-1}}f_{i+1}\cdots f_k\cdot (0) \\
&= 0.
\end{align*}
\end{proof}

The following lemma is the key result of this section, which allow us to construct universal Verma-type modules for $MU(n)$ in the next section.

\begin{lemma} \label{lem:maincalc}
Let $M$ be any $MU(n)$-module and let $x\in M$ be a mirabolic highest weight vector with weight $\mu$ and depth $r$. Furthermore, suppose $x$ has maximal depth amongst all vectors with weight $\mu$ in $M$. Then,
\[e_{r}e_{r - 1}\cdots e_2e_1 \ell f_1 f_2\cdots f_{r-1} f_{r} \cdot x\]
is a scalar multiple of $x$, where the scalar is determined by $\mu$ and $r$ and independent of $M$.
\end{lemma}

\begin{proof}
When $r=0$, we have $\ell\cdot x=x$, which proves the claim. Henceforth, we suppose that $r\geq 1$ and so $\ell\cdot x=0$. Define $x'=e_{r} e_{r - 1}\cdots e_2 e_1 \ell f_1 f_2\cdots f_{r-1} f_{r}\cdot x$. We will compute $\ell f_1 f_2\cdots f_{r-1} f_{r} \cdot x'$ and show that it is a scalar multiple of $\ell f_1 f_2\cdots f_{r-1} f_{r} \cdot x$. To that end, we have:
\[ \ell f_1 f_2\cdots f_{r-1} f_{r} \cdot x' = \ell f_1\cdots f_{r-1} f_{r} e_{r} e_{r - 1}\cdots e_1 \ell f_1\cdots f_{r-1} f_{r} \cdot x.\]
Between the 2 occurrences of the element $\ell$ in the above expression, we move each of the $f_j$'s to the right of each of the $e_i$'s by using the relation~\eqref{eq:usln.3}. (In this process, we accumulate some $k_i$'s and $k_i^{-1}$'s, which we can move to rightmost end and absorb into $x$ up to a multiplication by scalars.) Thus, we see that $\ell f_1 f_2\cdots f_{r-1} f_{r} \cdot x'$ is equal to a linear combination of terms of the form:
\[\ell e_{i_k}\cdots e_{i_2} e_{i_1} f_{i_1} f_{i_2}\cdots f_{i_k}  \ell f_1 f_2\cdots f_{r-1} f_{r} \cdot x,\]
for some $i_j$ such that $1\leq i_1< i_2<\cdots <i_k\leq r$. If $k=0$, we are done as we just get $\ell^2 f_1 f_2\cdots f_{r-1} f_{r} \cdot x=\ell f_1 f_2\cdots f_{r-1} f_{r} \cdot x$.

Next, suppose $k\neq 0$ and let $i_k>1$. Then, the latter half of the above term looks like:
\[\cdots f_{i_k} \ell  f_1 f_2\cdots f_{r-1} f_{r} \cdot x.\]
We can move the $f_{i_k}$ to the right till we get:
\[\cdots \ell  f_1 f_2\cdots f_{i_k}  f_{i_{k}-1}  f_{i_k} \cdots f_{r} \cdot x.\]
We use the identity $(v+v^{-1})f_{i_k}  f_{i_{k}-1}  f_{i_k} =f_{i_k}^2  f_{i_k - 1} + f_{i_k -1} f_{i_k}^2$ to get a sum involving two terms. In the first sum, we can move the $f_{i_k}^2$ to the left, till we have moved it to the left of $\ell$ to get:
\[\cdots f_{i_k}^2  \ell  f_1 f_2\cdots f_{i_k-2}  f_{i_{k}-1}  f_{i_k+1} \cdots f_{r} \cdot x.\]
Then, we can move the $\ell  f_1 f_2\cdots f_{i_{k}-1}$ to the right past all the other $f_j$'s and act on $x$ to get zero as $x$ has depth $r$. So, the only term that survives is the latter term, which is of the form:
\[\ \cdots \ell  f_1 f_2\cdots f_{i_k-1}  f_{i_{k}}^2  f_{i_k+1} \cdots f_{r} \cdot x.\]
Next, we repeat the same procedure by bringing $f_{i_{k-1}}$, which is the element to the left of $\ell$, to its right, exactly as we did with $f_{i_k}$. We similarly use the Serre relation as above, and the only term that survives is the one of the form:
\[\cdots \ell  f_1 f_2\cdots f_{i_{k-1}-1}  f_{i_{k-1}}^2  f_{i_{k-1}+1}\cdots f_{i_k-1}  f_{i_{k}}^2  f_{i_k+1} \cdots f_{r} \cdot x.\]
Thus, we can bring each of the $f_{i_j}$'s to the right of $\ell$ and the only term that survives is the one where the $f_{i_j}$ to the right of $\ell$ gets replaced by $f_{i_j}^2$ for all $j$. The only case where this is not immediately clear is when $i_1=1$, because then $f_{i_1}$ doesn't commute with $\ell$. In that case, we have an expression of the form:
\[\cdots f_1 \ell  f_1 \cdots.\]
We then use the identity $(v+v^{-1})f_1 \ell  f_1 = v^{-1}\ell f_1^2 + vf_1^2 \ell$. In the latter term, we can move the $\ell$ past all the $f_j$'s on its right to have it act on $x$, which gives zero. Hence, the only term that can survive is the one involving $\ell f_1^2$. Thus, we get an expression of the form:
\[\ell e_{i_k}\cdots e_{i_2} e_{i_1}  \ell  f_1^{\eps_1}  f_2^{\eps_2} \cdots f_r^{\eps_r} \cdot x,\]
where:
\[\eps_t=\begin{cases}2 & \text{if}~t=i_j~\text{for some}~j,\\1 &\text{otherwise}.\end{cases}\]
By Lemma~\ref{lem:reduce}, we get that the above expression is equal to:
\[\ell e_{i_k}\cdots e_{i_2} e_{i_1}  f_1^{\eps_1}  f_2^{\eps_2} \cdots f_r^{\eps_r} \cdot x.\]
Now, we move the $e_{i_j}'s$ to the right one by one, starting with $e_{i_1}$. When $e_{i_j}$ comes to $f_{i_j}^2$, we get a linear combination of terms where it commutes past it and terms where we get an expression involving $k_{i_j}^{\pm 1}$. The former term vanishes as $e_{i_j}$ will eventually move to rightmost end, where it acts on $x$ by zero. In the latter term, the $f_{i_j}^2$ gets replaced by $f_{i_j}$, and the $k_{i_j}^{\pm1}$ terms can be moved to the rightmost end and be absorbed by $x$, up to a scalar factor. Hence, up to a scalar factor depending only upon the weight $\mu$ and $r$, we end up with the term:
\[\ell f_1 f_2\cdots f_r\cdot x.\]
This shows that $\ell f_1 f_2\cdots f_{r-1} f_{r} \cdot x' = c \ell f_1 f_2\cdots f_{r-1} f_{r} \cdot x$ for some scalar $c\in K$. Let $\ol{x}=x-cx'$. We have shown that $\ell f_1 f_2\cdots f_{r-1} f_{r} \cdot \ol{x}=0$. We claim that, in fact, $\ell f_1 f_2\cdots f_{t-1} f_{t} \cdot \ol{x} =0$ for all $0\leq t\leq r$. This will show that $\ol{x}$ has depth at least $r+1$, which is possible only if $\ol{x}=0$ since $x$ was assumed to have maximal depth amongst all vectors with weight $\mu$. This would complete the proof of the lemma.

In order to show that $\ell f_1 f_2\cdots f_{t-1} f_{t} \cdot \ol{x} =0$ for any $0<t<r$ (we handle $t=0$ at the end), we only need to show that $\ell f_1 f_2\cdots f_{t-1} f_{t} \cdot x' =0$ as we know that $x$ has depth $r$. The element $\ell f_1 f_2\cdots f_{t-1} f_{t} \cdot x'$ is equal to:
\[\ell f_1\cdots f_{t-1} f_{t} e_{r} e_{r - 1}\cdots e_1 \ell f_1\cdots f_{r-1} f_{r} \cdot x.\]
From the middle part of the above expression, we can move $e_{r} e_{r-1}\cdots e_{t+1}$ to the leftmost end, to get:
\[e_{r} e_{r-1}\cdots e_{t+1}  \ell f_1\cdots f_{t-1} f_{t} e_{t} e_{t - 1}\cdots e_1 \ell f_1\cdots f_{r-1} f_{r} \cdot x.\]
Now, we can repeat the proof of the first half of this lemma word-by-word to show that the above expression is a scalar multiple of the term:
\[e_{r} e_{r-1}\cdots e_{t+1} \ell f_1\cdots f_{r-1} f_{r} \cdot x.\]
We bring the $e_{t+1}$ term to the right of $\ell$ where it commutes past everything (since $t\geq 1$) till it gets to $f_{t+1}$. Then, this expression is a linear combination of two terms. In the first term, the $e_{t+1}$ commutes past $f_{t+1}$ and moves to the rightmost end, where it acts on $x$ by zero. In the other term, the $e_{t+1} f_{t+1}$ gets replaced by an expression involving $k_{t+1}^{\pm1}$. Then, we have that $\ell f_1  f_2\cdots f_t$ commutes past this term and all the following $f_j$'s on its right to act on $x$, which equals zero as $x$ has depth $r>t$. Thus, we get that the above expression evaluates to zero.

Finally, we deal with the case $t=0$. In this case, we need to show that
\[\ell\cdot x' = \ell e_{r} e_{r- 1}\cdots e_1 \ell f_1\cdots f_{r-1} f_{r} \cdot x = \ell e_{r} e_{r - 1}\cdots e_1 f_1\cdots f_{r-1} f_{r} \cdot x\]
is zero. (Here, we have used Lemma~\ref{lem:reduce} for the second equality.) To show this, we use a similar strategy as the rest of the proof where we move each $e_j$ to the right, where it can either commute past $f_j$ or erase it up to a scalar multiple. If it moves past, $e_j$ can be moved to the rightmost end, where it acts on $y$ by zero. If for all $j$, we have that $e_j$ erases the $f_j$, we are eventually left with a scalar multiple of $\ell\cdot x=0$. This completes the proof.
\end{proof}

The proof of the above lemma implies the following corollaries:

\begin{corollary} \label{cor:depth-increase}
Let $N$ be an arbitrary $MU(n)$-module and let $y\in N$ be a mirabolic highest weight vector with depth $r<n$. Consider the vector:
\[y'=e_re_{r-1}\cdots e_2e_1\ell f_1f_2\cdots f_{f-1}f_r\cdot y.\]
Then, there exists a scalar $c\in K$ such that the (possibly zero) vector $y-cy'$ has depth $>r$.
\end{corollary}

\begin{corollary} \label{cor:fmove}
Let $N$ be an arbitrary $MU(n)$-module and let $y\in N$ be a mirabolic highest weight vector with depth $r$. Consider the vector:
\[y':= f_{i_1} f_{i_2}\cdots f_{i_k}  \ell f_1 f_2\cdots f_{s-1} f_{s} \cdot y,\]
for some $i_j$ such that $1\leq i_1< i_2<\cdots <i_k\leq r$ and some $s\geq r$. Then, we have the equality:
\[y'=  \frac{c}{(v+v^{-1})^k}\ell  f_1^{\eps_1}  f_2^{\eps_2} \cdots f_s^{\eps_s} \cdot y,\]
where:
\[\eps_t=\begin{cases}2 & \text{if}~t=i_j~\text{for some}~j,\\1 &\text{otherwise},\end{cases}\]
and $c=v^{-1}$ if $i_1=1$ and $c=1$ otherwise.
\end{corollary}

We end this section with another corollary that computes the scalar alluded to in Lemma~\ref{lem:maincalc}.
\begin{corollary} \label{cor:char-chi}
Fix a dominant integral weight $\lambda$ and an integer $0\leq r\leq n$ and let $\mu=\lambda+\omega_r$. Let $M$ be an $MU(n)$-module and let $x\in M$ be a mirabolic highest weight vector with weight $\mu$ and depth $r$. Furthermore, suppose $x$ has maximal depth amongst all vectors with weight $\mu$ in $M$. Then,
\[e_{r}e_{r - 1}\cdots e_2e_1 \ell f_1 f_2\cdots f_{r-1} f_{r} \cdot x=v^{\lambda_{r+1}-\lambda_1}x.\]
\end{corollary}

\begin{proof}
Consider the $MU(n)$-representation $L_{\lambda, r}$ and the mirabolic highest weight vector $x_{\lambda, r}\in L_{\lambda, r}$. By Lemma~\ref{lem:maincalc}, we can suppose that $M=L_{\lambda, r}$ and $x=x_{\lambda, r}$. Then, we see that:
\begin{align*}
e_r\cdots e_2e_1\ell f_1f_2\cdots f_r\cdot x_{\lambda, r} &= e_r\cdots e_2e_1\ell f_1f_2\cdots f_r\cdot (x_{\lambda}\otimes u_{\{1,2,\dots, r\}}) \\
&=e_r\cdots e_2e_1 \cdot((k_1^{-1}k_2^{-1}\cdots k_{r}^{-1}\cdot x_{\lambda})\otimes(f_1f_2\cdots f_{r}\cdot u_{\{1,2,\dots, r\}}))\\
&=v^{\lambda_{r+1}-\lambda_1}e_r\cdots e_2e_1 \cdot( x_{\lambda}\otimes u_{\{2,3,\dots, r+1\}})\\
&=v^{\lambda_{r+1}-\lambda_1} x_{\lambda}\otimes(e_{r}\cdots e_2e_1\cdot u_{\{2,3,\dots, r+1\}})\\
&=v^{\lambda_{r+1}-\lambda_1} x_{\lambda}\otimes u_{\{1,2,\dots, r\}}.
\end{align*}
\end{proof}

\section{Verma module type constructions} \label{sec:Verma}

\subsection{Definition of the modules}

In this section, we state two constructions of infinite dimensional representations of the algebra $MU(n)$ that are inspired by Verma modules over $\uvsln$ and will be used in the next section for classifying finite dimensional representations.

\begin{definition}Fix a dominant integral weight $\lambda$ and an integer $0\leq r\leq n$. Let $\mu=\lambda+\omega_r$. Viewing the space $MU(n)$ as a left module over itself, consider the left submodule $N_{\lambda, r}\sub MU(n)$ generated by the elements:
\[e_i,~k_i -v^{\mu_i-\mu_{i+1}}\qquad 1\leq i\leq n-1;\]
\[\ell f_1f_2\cdots f_{i},\quad0\leq i\leq r-1;\qquad e_{r}e_{r - 1}\cdots e_2e_1 \ell f_1 f_2\cdots f_{r-1} f_{r}-v^{\lambda_{r+1}-\lambda_1}.\]
We define the quotient module $V_{\lambda, r}:=MU(n)/N_{\lambda, r}$.
\end{definition}
There exists a canonical surjection $MU(n)\to V_{\lambda, r}$ of $MU(n)$-modules and we suppose the image of $1\in MU(n)$ under this surjection is the vector $w_{\lambda, r}\in V_{\lambda, r}$. It is clear that $w_{\lambda, r}$ is a mirabolic highest weight vector with weight $\mu$ and depth $r$, and that it generates $V_{\lambda, r}$ as a left $MU(n)$-module.

\begin{remark}
Consider the subalgebra $MU(n)^+\sub MU(n)$ generated by all the $e_i$'s and the degree zero subspace $MU(n)_{\bold{0}}$ with respect to the grading in Remark~\ref{rem:grading}. Then, it is possible to view the representation $V_{\lambda, r}$ as being induced from a one-dimensional representation of $MU(n)^+$, but we won't be needing that perspective here.
\end{remark}

We can construct another universal $MU(n)$-representation more directly from the theory of $\uvsln$-representations: for any weight $\lambda$, there exists a universal highest weight $\uvsln$-representation, or Verma module, $M_{\lambda}$. If $\lambda$ is an integral dominant weight, then $M_{\lambda}$ 
has a unique finite dimensional quotient, which is isomorphic to 
the simple highest weight module $L_{\lambda}$.

\begin{definition}
We define the $MU(n)$-module $M_{\lambda,r}: = M_{\lambda} \otimes W_r$, where the action of the algebra $MU(n)$ is via the comodule map $\rho$.
\end{definition}
When $\lambda$ is a dominant integral weight, the simple module $L_{\lambda, r}$ is a finite dimensional simple quotient of $M_{\lambda,r}$. The module $M_{\lambda,r}$ has a unique (up to scaling) vector with weight $\lambda+\omega_r$ which is given by $x_{\lambda,r}:=x_{\lambda}\otimes u$, where $x_{\lambda}$ is the unique (up to scaling) highest weight vector of the $\uvsln$-module $M_{\lambda}$ with weight $\lambda$ and $u:=u_{\{1,\ldots,r\}}=u_1\wedge u_2\wedge\cdots\wedge u_r \in W_r$. We can repeat the proof of Lemma~\ref{lem:highcyclic} to show that $M_{\lambda,r}$ is generated as an $MU(n)$-module by the vector $x_{\lambda,r}$.

The two representations defined above are related by the following theorem:
\begin{theorem} \label{prop:Verma}
There exists an $MU(n)$-module isomorphism $\Theta: V_{\lambda,r} \xra{\sim} M_{\lambda,r}$ such that $\Theta(w_{\lambda,r})=x_{\lambda,r}$.
\end{theorem}
\begin{proof}
There exists a surjective map of left $MU(n)$-modules $\Theta:MU(n)\to M_{\lambda,r}$, $1\mapsto x_{\lambda,r}$. Since the highest weight vector $x_{\lambda,r}$ has weight $\mu$ and depth $r$, by the definition of the submodule $N_{\lambda, r}$ and by Corollary~\ref{cor:char-chi}, this map descends to the quotient $\Theta:MU(n)/N_{\lambda, r}\to M_{\lambda,r}$. We will show that it is an isomorphism in Section~\ref{sec:proof-Verma} by providing a construction of an inverse map.
\end{proof}
As a corollary of this theorem, we see that the representation $V_{\lambda,r}$ has depth exactly $r$ and has a unique vector (up to scaling) with weight $\lambda+\omega_r$.

Next, we prove some properties about the representation $M_{\lambda, r}$ when $\lambda$ is a dominant integral weight. In particular, we show that it has a unique non-zero finite dimensional quotient.

\begin{lemma} \label{lem:Jantzen}
Suppose $\lambda$ is a dominant integral weight and $V=M_{\lambda,r}/I$ is a finite dimensional quotient of the module $M_{\lambda,r}$, where $I$ is some left submodule of $M_{\lambda,r}$. Then, for all $i$, we have that $(f_i^{m_i+1}\cdot x_{\lambda})\otimes u \in I$, where $m_i=(\lambda,\alpha_i)$.
\end{lemma}

\begin{proof}
When $r=0$ or $r=n$, $M_{\lambda,r}\simeq M_{\lambda}$ as a $\uvsln$-module, and so, the lemma is a consequence of \cite[Lemma 5.4(b)]{J}. Henceforth, we suppose that $1\leq r\leq n-1$.

For any $f_i$ such that $i\neq r$, we have that $f_i\cdot u=0$ and so, $f_i^t\cdot (x_{\lambda}\otimes u) = (f_i^t\cdot x_{\lambda})\otimes u$ for all $t$. Then, by \cite[Lemma 5.4(b)]{J}, we get that $(f_i^{m_i+1} \cdot x_{\lambda})\otimes u=f_i^{m_i+1} \cdot (x_{\lambda}\otimes u) \in I$

Next, consider the action of $f_r$ on $x_{\lambda}\otimes u$. Note that the weight of $x_{\lambda}\otimes u$ with respect to $k_r$ is $(\lambda+\omega_r,\alpha_r)=m_r+1$. Hence, again by \cite[Lemma 5.4(b)]{J}, we get that $f_r^{(m_r+1)+1} \cdot (x_{\lambda}\otimes u)=f_r^{m_r+2} \cdot (x_{\lambda}\otimes u) \in I$. We'll use this to show that $(f_r^{m_r+1}\cdot x_{\lambda})\otimes u\in MU(n)\cdot f_r^{m_r+2}\cdot (x_{\lambda}\otimes u)$, which will show that $(f_r^{m_r+1}\cdot x_{\lambda})\otimes u$ is in $I$.

Recall that $f_r$ acts on the tensor product $M_{\lambda}\otimes W_r$ via $\rho(f_r)=k_r^{-1}\otimes f_r + f_r \otimes 1$. Also, it is clear that $f_r\cdot u = u_1\wedge u_2\wedge\cdots \wedge u_{r-1}\wedge u_{r+1} =: u'$ and that $f_r^2\cdot u =0$. Thus, $f_r^{m_r+2}\cdot (x_\lambda\otimes u)$ is equal to:

\begin{align}
\notag f_r^{m_r+2}\cdot (x_\lambda\otimes u)&=(f_r^{m_r+2} \cdot x_{\lambda})\otimes u+\sum_{i=0}^{m_r+1} (f_r^{m_r-i+1}k_r^{-1}f_r^{i} \cdot x_{\lambda})\otimes u'\\
\label{eq:fmr-dot-xu} &=(f_r^{m_r+2} \cdot x_{\lambda})\otimes u + c_r(f_r^{m_r+1} k_r^{-1} \cdot x_{\lambda})\otimes u',    
\end{align}
where $c_r=1+v^2+\cdots +v^{2(m_r+1)}$.

Next, we act on this element by the product $\ell f_1 f_2\cdots f_{r-1}$. For $i<r$, it is clear that $f_i\cdot u=0$. Therefore, 
\[\ell f_1f_2\cdots f_{r-1} \cdot ((f_r^{m_r+2} \cdot x_{\lambda})\otimes u) = \ell\cdot ((f_1 f_2\cdots f_{r-1} f_r^{m_r+2} \cdot x_{\lambda})\otimes u),\]
which is zero because $\ell\cdot u=0$. Hence, the first term of \eqref{eq:fmr-dot-xu} 
doesn't survive when we act by $\ell f_1 f_2\cdots f_{r-1}$. Next, when we act on the second term, each $f_i$ acts via $\rho(f_i)=k_i^{-1}\otimes f_i + f_i \otimes 1$. Here, we can ignore the $f_i\otimes 1$ part, because that will not survive in the end when we act by $\ell$, and so, up to a scalar multiple, we get:
\begin{equation}\label{eq:lffr-dot-xu}\ell f_1 f_2\cdots f_{r-1}f_r^{m_r+2}\cdot (x_\lambda\otimes u)=c(f_r^{m_r+1} \cdot x_{\lambda}) \otimes u'',\end{equation}
where $u''=u_2\wedge u_3\wedge\cdots \wedge u_{r+1}$ and $c\in K$ is determined by the equation $k_1^{-1}k_2^{-1}\cdots k_{r}^{-1}(f_r^{m_r+1} \cdot x_{\lambda})=cf_r^{m_r+1} \cdot x_{\lambda}$. (When $r=1$, we have that $u''=u'$.) Next, we act on \eqref{eq:lffr-dot-xu} 
by $e_r e_{r-1}\cdots e_1$. Each $e_i$ acts on the tensor product via $\rho(e_i)=1\otimes e_i + e_i\otimes k_i$. Next, we note that for all $i$, we have that $e_i\cdot( f_r^{m_r+1} \cdot x_{\lambda}) = 0$. This is clear when $i\neq r$ since $e_i$ commutes with $f_r$ and $x_{\lambda}$ is a highest weight vector. When $i=r$, this follows from the computation in \cite[Lemma 5.6]{J}. So, we get that 
\begin{align*}e_r e_{r-1}\cdots e_1\cdot ((f_r^{m_r+1} \cdot x_{\lambda}) \otimes u'') &= (f_r^{m_r+1} \cdot x_{\lambda}) \otimes (e_r e_{r-1}\cdots e_1\cdot u'')\\&= (f_r^{m_r+1} \cdot x_{\lambda}) \otimes u,\end{align*}
completing the proof.
\end{proof}

\begin{proposition}\label{prop:unique}
Suppose $\lambda$ is a dominant integral weight and $V=M_{\lambda,r}/I$ is a non-zero finite dimensional quotient of the module $M_{\lambda,r}$, where $I$ is some left submodule of $M_{\lambda,r}$. Then, we have the equality $V=L_{\lambda, r}$.
\end{proposition}

\begin{proof}
By the previous lemma, we know that the submodule $I$ contains the elements $x_i=(f_i^{m_i+1}\cdot x_{\lambda})\otimes u$ for all $i$. Let $I'$ by the $\uvsln$-submodule of $M_{\lambda}$ generated by $f_i^{m_i+1}\cdot x_{\lambda}$ for all $i$, where $m_i=(\lambda,\alpha_i)$. Then, by Proposition 5.9 and Theorem 5.15 of \cite{J}, we have that $M_{\lambda}/I' = L_{\lambda}$.

We note that $e_j$ acts by zero on $f_i^{m_i+1}\cdot x_{\lambda}$ for all $i$ and $j$. This is clear when $i\neq j$ since $e_j$ commutes with $f_i$. When $i=j$, this follows from the computation in \cite[Lemma 5.6]{J}. Therefore, the vector $f_i^{m_i+1}\cdot x_{\lambda}$ is a highest weight vector. Then, repeating the argument in the proof of Lemma~\ref{lem:highcyclic}, we get that $(f_i^{m_i+1}\cdot x_{\lambda}) \otimes W_r \sub I$ for all $i$. This, in turn, implies that $I' \otimes W_r\sub I$. Therefore, we have the following surjection of $MU(n)$-modules:
\[L_{\lambda,r} = L_{\lambda} \otimes W_r = \frac{M_{\lambda}}{I'} \otimes W_r  =\frac{M_{\lambda} \otimes W_r}{I'\otimes W_r} \twoheadrightarrow \frac{M_{\lambda} \otimes W_r}{I}=\frac{M_{\lambda,r}}{I} =  V.\]
As $L_{\lambda,r}$ is simple and $V\neq 0$, we get that $V=L_{\lambda,r}$.
\end{proof}

As a consequence, we obtain 
an explicit presentation for the simple representations $L_{\lambda, r}$.

\begin{corollary} \label{cor:exppresent}
Fix a dominant integral weight $\lambda$  and $0\leq r\leq n$. Let $\mu=\lambda+\omega_r$ and let $n_i=(\mu, \alpha_i)$. Consider the left ideal $J$ of $MU(n)$ generated by the following elements:
\[e_i,~k_i -v^{\mu_i-\mu_{i+1}}, ~f_i^{n_i+1}\qquad 1\leq i\leq n-1;\]
\[\ell f_1f_2\cdots f_{i},\quad0\leq i\leq r-1;\qquad e_{r}e_{r - 1}\cdots e_2e_1 \ell f_1 f_2\cdots f_{r-1} f_{r}-v^{\lambda_{r+1}-\lambda_1}.\]
Then, we have that $MU(n)/J \simeq L_{\lambda, r}$ as left $MU(n)$-modules.
\end{corollary}

\begin{proof}
This is a consequence of the isomorphism $V_{\lambda, r}\simeq M_{\lambda, r}$ and the proofs of Lemma~\ref{lem:Jantzen} and Proposition \ref{prop:unique}.
\end{proof}

\subsection{Conclusion of proof of Theorem~\ref{prop:Verma}\label{sec:proof-Verma}}
Recall that 
$V_{\lambda,r}$ was constructed as a quotient of 
$MU(n)$ and has a highest weight vector $ w_{\lambda, r}$. The representation $M_{\lambda, r}=M_{\lambda}\otimes W_r$ is generated by the vector $x_{\lambda, r}$ which has weight $\mu:=\lambda+\omega_r$ and depth $r$, and is the unique vector in its weight space (up to scalar multiples). We have already shown that
there exists a surjection $\Theta:V_{\lambda,r} \to M_{\lambda,r}$ that sends $w_{\lambda, r}$ to $x_{\lambda, r}$. In this section, we construct an inverse to this map by mimicking the proof of Lemma~\ref{lem:highcyclic}, hence showing that it is an isomorphism.

\begin{proposition}
The map $\Theta:V_{\lambda,r} \to M_{\lambda,r}$ has an inverse 
when $r=0$ or $r=n$.
\end{proposition}

\begin{proof}
Suppose $r=0$. In that case, $\ell$ acts by $1$ on $M_{\lambda, 0}$. Furthermore, note that the space $V_{\lambda,0}$ is spanned by elements of the form $m\cdot  w_{\lambda, 0}$ where $m\in MU(n)$. Then, by the construction of $V_{\lambda,r}$, we get that $\ell\cdot w_{\lambda, 0}= w_{\lambda,0}$. We claim that $\ell$ acts by $1$ on all of $V_{\lambda, 0}$. To see this, pick a non-commutative monomial $m\in MU(n)$ having minimal degree in the generators, such that $\ell\cdot (m\cdot w_{\lambda, r})\neq m\cdot w_{\lambda, r}$. By minimality, we can assume that $m$ does not consist of $\ell$ or any of the $k_i^{\pm1}$'s. Furthermore, it can't consist of any of the $e_i$'s either, since they can be commuted past the $f_j$'s, up to lower degree terms, to eventually act on $w_{\lambda, 0}$ by $0$. Thus, we have that $m=Q(f)$, a non-commutative monomial in the $f_j$'s. But then, by Lemma~\ref{lem:reduce} we conclude that $\ell\cdot(m\cdot w_{\lambda, 0})= \ell \cdot(m\ell\cdot w_{\lambda, 0}) = \ell m\ell \cdot w_{\lambda, 0} = m\ell\cdot w_{\lambda, 0} =m\cdot w_{\lambda, 0}$. Thus, $\ell$ acts by $1$ on both $M_{\lambda,0}$ and $V_{\lambda, 0}$, and so, they are both pullbacks to $MU(n)$ of $\uvsln$-representations via the map $\pi_1:MU(n)\to \uvsln$. As $M_{\lambda, 0}=\pi_1^*(M_{\lambda})$, by the universal property of Verma modules, there exists a map $M_{\lambda, 0}\to V_{\lambda, 0}$ that maps $x_{\lambda, 0}$ to $w_{\lambda, 0}$. Since $M_{\lambda, 0}$ is generated as an $MU(n)$-module by $w_{\lambda, 0}$, this map acts as an inverse to $\Theta$, which proves the required claim.

Next, let $r=n$, in which case $\ell$ acts by $0$ on $M_{\lambda, n}$. 
We claim that $\ell$ acts by $0$ on $V_{\lambda, n}$. To see this, pick a non-commutative monomial $m\in MU(n)$ having minimal degree in the generators, such that $\ell\cdot (m\cdot w_{\lambda, r})\neq 0$. By the same argument as above, we can assume that $m=Q(f)$, a non-commutative monomial in the $f_j$'s. Then, by Corollary~\ref{cor:lastcor}, we can express $\ell Q(f)$ as a linear combination of terms of the form $Q_s'(f) \ell f_1f_2\cdots f_s$ for varying $s$ and some non-commutative polynomials $Q_s'(f)$. But, by construction, the element $\ell f_1f_2\cdots f_s$, for all $0\leq s\leq n-1$, acts on $w_{\lambda, r}$ by zero, proving the required claim. Hence, both $M_{\lambda, n}$ and $V_{\lambda, n}$ are pullbacks of representations via the map $\pi_0: MU(n)\to \uvsln$, and so, we can repeat the argument above.
\end{proof}

Henceforth, we assume that $1\leq r\leq n-1$.  
\begin{remark}
We state a slight reformulation of the grading from Remark~\ref{rem:grading}.  The algebra $K[k_1^{\pm1}, k_2^{\pm1},\cdots, k_{n-1}^{\pm1}]$ acts on $MU(n)$, where each $k_i$ acts by conjugation, i.e. $k_i\cdot m=k_imk_i^{-1}$ for all $m\in MU(n)$. By the definition of $MU(n)$ in terms of generators and relations, this action is diagonalizable, hence we can apply the same theory of weights as in Section~\ref{sec:weightspaces}. Then we have a weight space decomposition $MU(n)=\oplus_{\mu\in \ZZ^n/\ZZ}MU(n)_\mu$, where $m\in MU(n)_\mu$ if $k_imk_i^{-1} = v^{\mu_{i}-\mu_{i+1}}m$ for all $1\leq i\leq n-1$. It is clear that this respects the algebra structure, i.e. $MU(n)_\mu\cdot MU(n)_\eta\subset MU(n)_{\mu+\eta}$, and hence defines a grading. Further, the weights of the $k_i$'s and $\ell$ are $(0,\ldots,0)$, and the weight of $e_i$ (resp. $f_i$) is the simple root $\alpha_i$ (resp. $-\alpha_i$), hence $MU(n)$ is actually graded by the root lattice. The $\ZZ^{n-1}$-grading in Remark~\ref{rem:grading} is then related to the weight grading via the isomorphism of abelian groups between the root lattice and $\ZZ^{n-1}$ given by $\alpha_i\mapsto (0,\ldots,0,\stackrel{i}{1},0,\ldots,0)$. An element of the root lattice has a unique representative $\mu$ such that $\sum\mu_i=0$, hence another way to write the isomorphism between the root lattice and the $\ZZ^{n-1}$-grading is via the map, $(\mu_1,\ldots,\mu_{n-1},-\sum_{i=1}^{n-1}\mu_i)\mapsto(\mu_1,\mu_1+\mu_2,\ldots, \mu_1+\cdots+\mu_{n-1})$.
\end{remark}

Note that $M_{\lambda,r}$ is spanned by elements of the form $(m\cdot x_{\lambda})\otimes u_I$ where $m\in MU(n)$ and $I\in [n]_r$. For simplicity of notation in the computations that follow, we will denote $(m\cdot x_{\lambda})\otimes u_I$ 
by the pair $(m,I)\in MU(n)\times [n]_r$. In particular, we have the highest weight vector $x_{\lambda,r}=(1\cdot x_{\lambda})\otimes u_{\{1,\ldots,r\}}=(1,\{1,\ldots,r\})$. Note that $(k_i, I) = v^{\lambda_i-\lambda_{i+1}}(1,I)$ for all $i$.

For any $m'\in MU(n)$, we will denote by $m'\cdot (m,I)$ the element $m'\cdot ((m\cdot x_{\lambda})\otimes u_I)$.  Choose a basis $\mathcal B$ of the negative part of the quantum group $\uvsln$, consisting of non-commutative polynomials in the $f_i$'s. Then, a basis for $M_{\lambda,r}$ is given by pairs $(Q(f), I)\in\bb\times  [n]_r$.

In the rest of the section, we will use the following equality extensively which follows from the fact that $e_j\cdot x_{\lambda}=0$: For any $e_j, Q(f)$ and $I$, we have that:
\[e_j\cdot (Q(f)\cdot x_{\lambda}) = [e_j, Q(f)]\cdot x_{\lambda}.\]

We construct an inverse map $\Phi:M_{\lambda,r}\to V_{\lambda,r}$ layer-by-layer. We start by first defining $\Phi$ on all elements of the form $(Q(f),I)$ where $I\sub \{1,2,\dots,r+1\}$. Define $I_i:=\{1,2,\dots, r+1\}\sm\{i\}$. The construction will be by induction on the degree of $Q(f)$. First, suppose the degree of $Q$ is zero. Define 
\begin{equation}\label{eq:Phi1}\Phi(1,I_1) := v^{\lambda_1-\lambda_{r+1}}\ell f_1f_2\cdots f_r \cdot  w_{\lambda,r};\qquad
\Phi(1,I_i):=e_{i-1}e_{i-2}\cdots e_1\cdot \Phi(1,I_1), \text{ for }2\leq i\leq r+1.\end{equation}

We note that, by definition, 
\begin{align}\notag\Phi(x_{\lambda,r})=\Phi(1,\{1,\ldots,r\})& =\Phi(1,I_{r+1})=v^{\lambda_1-\lambda_{r+1}}e_re_{r-1}\cdots e_1\ell f_1\cdots f_{r-1}f_r\cdot w_{\lambda, r}\\
\label{eq:topvect}\text{ (by the definition of $V_{\lambda,r}$)} & = w_{\lambda,r}.
\end{align}

Now, suppose the map $\Phi$ has been constructed for a given $Q(f)$ having weight $\mu-\lambda$ and all polynomials of lower degree. (This choice of weight has been made so that the element $Q(f)\cdot x_{\lambda}$ has weight $\mu$.) In particular, we have defined $\Phi(Q(f),I_{r+1})$ and $\Phi(Q(f),I_1)$. Define
\begin{align}\label{eq:Phi-fj}\Phi(f_jQ(f),I_1)&:=f_j\cdot \Phi(Q(f),I_1),~\text{ for all }j\neq r+1; \\
\label{eq:Phi-fr}\Phi(f_{r+1}Q(f),I_{1})&:=v^{\mu_1-\mu_{r+1}+1}\ell f_1f_2\cdots f_{r+1} \cdot \Phi(Q(f),I_{r+1}).\end{align}
Thus, we have constructed $\Phi(\widetilde Q(f),I_1)$ whenever $\deg(\widetilde{Q}(f))\leq \deg(Q(f))+1$. Next, we define $\Phi(\widetilde Q(f),I_i)$ for $1\leq i\leq r+1$ by induction on $i$. The base case has been done. For $i\geq 1$, we define:
\begin{equation}\label{eq:Phi-I+1}\Phi(\widetilde Q(f),I_{i+1})=e_i\cdot \Phi(\widetilde Q(f), I_i)-v^{-1}\Phi([e_i,\widetilde Q(f)], I_{i}).\end{equation}
Note that $[e_i,\widetilde{Q}(f)]$ is a (possibly zero) non-commutative polynomial in the $f_j$'s and the $k_j^{\pm1}$'s, which can then be simplified to a polyomial in the $f_i$'s of degree less or equal to $\deg (\widetilde{Q}(f))$. This completes the definition of the map $\Phi$ whenever $I\sub \{1,2,\dots, r+1\}$.

Finally, we define the map $\Phi$ for general $I\sub [n]_r$ by induction on the lexicographic order: Consider $I=\{i_1,i_2,\dots, i_r\} \sub \{1,2,\dots, n\}$ and suppose $i_r>r+1$ (that is, $I\neq I_j$ for any $j$). Pick the largest index $j$ such that $i_j\in I$ but $i_j-1\not\in I$. Define $I'=(I\cup\{i_j-1\})\sm\{i_j\}$. It is clear that $I'<I$ in the lexicographic order. Then, we define for any $Q(f)$ with weight $\mu-\lambda$: 
\begin{equation}\label{eq:Phi-I'}\Phi(Q(f),I):=v^{\mu_{i_j-1}-\mu_{i_j}}f_{i_j-1}\cdot \Phi(Q(f),I') - v^{\mu_{i_j-1}-\mu_{{i_j}}}\Phi(f_{i_j-1}Q(f), I').\end{equation}
 This completes the definition of $\Phi$ on the specified basis of $M_{\lambda,r}$. Before stating the main proposition of this section, we prove a technical lemma that we will be needing.

\begin{lemma}\label{lem:Phi-ei}
For all $1\leq i\leq r$ and $Q(f)\in \bb$ we have the equality:
\[e_i\cdot \Phi(Q(f), (I_{i}\cup\{r+2\})\sm \{r+1\}) =  \Phi(e_i\cdot(Q(f), (I_{i}\cup\{r+2\})\sm \{r+1\})).\]
\end{lemma}

 \begin{proof}
We use the fact that $[e_{i},f_{r+1}]=0$ to see that:
\begin{align*}
&~e_{i}\cdot \Phi(Q(f), (I_{i}\cup\{r+2\})\sm \{r+1\})\\
\text{(by \eqref{eq:Phi-I'})}=&~v^{\mu_{r+1}-\mu_{r+2}}e_{i}\cdot (f_{r+1}\cdot\Phi(Q(f), I_{i})-\Phi(f_{r+1}Q(f),I_{i}))\\
=&~v^{\mu_{r+1}-\mu_{r+2}}f_{r+1}e_{i}\cdot  \Phi(Q(f), I_{i}) - v^{\mu_{r+1}-\mu_{r+2}}e_{i}\cdot \Phi(f_{r+1}Q(f),I_{i})\\
\text{(by \eqref{eq:Phi-I+1})}=&~v^{\mu_{r+1}-\mu_{r+2}}f_{r+1}\cdot (v^{-1}\Phi([e_{i},Q(f)], I_{i}) +\Phi(Q(f),I_{i+1}))\\&-v^{\mu_{r+1}-\mu_{r+2}}(v^{-1}\Phi(f_{r+1}[e_{i},Q(f)],I_{i})+\Phi(f_{r+1}Q(f),I_{i+1}))\\
\text{(by \eqref{eq:Phi-I'})}=&~v^{-1}\Phi([e_{i},Q(f)], (I_{i}\cup\{r+2\})\sm \{r+1\})+ \Phi(Q(f), (I_{i+1}\cup\{r+2\})\sm \{r+1\})\\
\text{(by \eqref{eq:Phi-I+1})}=&~\Phi(e_{i}\cdot(Q(f), (I_{i}\cup\{r+2\})\sm \{r+1\})),
\end{align*}
which gives us the desired equality.
 \end{proof}
 
\begin{proposition}
The $K$-linear map $\Phi:M_{\lambda,r}\to V_{\lambda,r}$ is a map of $MU(n)$-modules.
\end{proposition}

Since the modules $M_{\lambda, r}$ and $V_{\lambda,r}$ are generated as $MU(n)$-modules by the vectors $x_{\lambda, r}$ and $w_{\lambda, r}$ respectively and we have the equalities $\Theta(w_{\lambda, r})=x_{\lambda, r}$ and $\Phi(x_{\lambda, r})=\Phi(1,I_{r+1})=w_{\lambda, r}$ (by~\eqref{eq:topvect}), this proposition shows that $\Theta$ is an isomorphism with inverse $\Phi$.

\begin{proof}
We need to show that $\Phi(z\cdot(Q(f), I)) = z\cdot \Phi(Q(f), I)$ for any generator $z$ of $MU(n)$. When $z=k_j^{\pm 1}$, the claim is clear. The proof of the proposition has been broken down into two lemmas: we prove the claim when $I=I_i$ for some $i$ in Lemma~\ref{lem:I_r+1} and for general $I$ in Lemma~\ref{lem:general}.
\end{proof}

Before we prove the lemmas, we state a modified definition of the `degree' of a non-commutative monomial $Q(f)$, denoted by $\ol{\deg}(Q(f))$ that we use in the rest of this section: We stipulate that $\ol\deg(f_i)=1$ for all $i\neq r+1$ (this condition is vacuous when $r=n-1$) and $\ol\deg(f_{r+1})=2r+1$. Then, for any non-commutative monomial $Q(f)$, we can define $\ol{\deg}(Q(f))$ inductively, using the fact that the relation~\eqref{eq:usln.6} is homogeneous with respect to this definition.

\begin{lemma} \label{lem:I_r+1}
For any $1\leq i\leq r+1$ and any $Q(f)\in \mathcal B$, we have the equality
\[\Phi(z\cdot(Q(f), I_i)) = z\cdot \Phi(Q(f), I_i),\] where $z=\ell,e_j,f_j$ 
for $1\leq j\leq n-1$.
\end{lemma}

\begin{proof}
We will prove the lemma for pairs $(Q(f), I_i)$ via induction on $\ol\deg(Q(f))+i$. Throughout the rest of the proof, we will assume that $Q(f)$ has weight $\mu-\lambda$ with respect to the $k_j$'s. We start with the base case of the induction when $i=1$ and $Q(f)=1$. When $z=f_j$ and $j\neq r+1$, we have that:
\begin{equation}\label{eq:fj-11}\Phi(f_j\cdot (1, I_1)) = \Phi(f_j, I_1) = f_j\cdot \Phi(1,I_{1}),\end{equation}
where the first equality follows from the definition of $\rho$ and second equality is by \eqref{eq:Phi-fj}. 
When $j=r+1$, 
\begin{align}
\notag\Phi(f_{r+1}\cdot (1, I_1)) &= \Phi(f_{r+1}, I_1) + v^{\lambda_{r+2}-\lambda_{r+1}}\Phi(1,(I_1\cup\{r+2\})\sm\{r+1\})\\
\label{eq:fr-11}\text{(by \eqref{eq:Phi-I'})}&= f_{r+1}\cdot \Phi(1,I_{1}).
\end{align}
Next, we suppose $z=e_j$ for some $j$. If $j=1$, we have:
\[\Phi(e_1\cdot (1,I_1))  = \Phi(1,I_2)\\
 = e_1\cdot \Phi(1,I_1),\]
where the second equality is by \eqref{eq:Phi1}. Now, we suppose that $j\neq 1$. Note that $e_j\cdot (1,I_1)=0$. On the other hand, if $j>r$ we have
\[e_j\cdot\Phi(1,I_1)= v^{\lambda_1-\lambda_{r+1}}e_j\ell f_1f_2\cdots f_r\cdot w_{\lambda,r}=v^{\lambda_1-\lambda_{r+1}}\ell f_1f_2\cdots f_re_j\cdot w_{\lambda,r}=v^{\lambda_1-\lambda_{r+1}}\ell f_1f_2\cdots f_r\cdot( e_j\cdot w_{\lambda,r})=0;\]
and, when $2\leq j\leq r$, we also have that
\begin{align*}
e_j\cdot\Phi(1,I_1)&= v^{\lambda_1-\lambda_{r+1}}e_j\ell f_1f_2\cdots f_r\cdot w_{\lambda,r}=v^{\lambda_1-\lambda_{r+1}}\ell e_j f_1f_2\cdots f_r\cdot w_{\lambda,r}\\
&=v^{\lambda_1-\lambda_{r+1}}\ell f_1f_2\cdots f_re_j\cdot w_{\lambda,r} + v^{\lambda_1-\lambda_{r+1}}\ell f_1f_2\cdots f_{j-1}\frac{k_j-k_j^{-1}}{v-v^{-1}}f_{j+1\cdots}f_r\cdot w_{\lambda,r}\\
&=0,
\end{align*}
where the last equality follows since $e_j$ and $\ell f_1f_2\cdots f_{j-1}$ act on $w_{\lambda, r}$ by zero. Finally, we note that $\ell$ acts on $\Phi(1,I_1) = v^{\lambda_1-\lambda_{r+1}}\ell f_1f_2\cdots f_r \cdot w_{\lambda, r}$ by $1$, proving equivariance with respect to $\ell$ and completing the proof of the base case of the induction.

Next we proceed with the induction step. First, we suppose that $i=1$. The fact that $f_j\cdot\Phi(Q(f),I_1)=\Phi(f_j\cdot(Q(f),I_1))$ follows from exactly the same calculation as when $Q(f)=1$ in \eqref{eq:fj-11} and \eqref{eq:fr-11}. Next, we suppose $z=e_j$ for some $j$. If $j=1$, we have:
\begin{align*}\Phi(e_1\cdot (Q(f),I_1)) & = \Phi(Q(f),I_2) + v^{-1}\Phi(e_1Q(f),I_1) \\
&= \Phi(Q(f),I_2) + v^{-1}\Phi([e_1,Q(f)],I_1) \\
\text{(by \eqref{eq:Phi-I+1})} &= e_1\cdot \Phi(Q(f),I_1).\end{align*}
Now, we suppose that $j\neq 1$. If $\ol\deg(Q(f))\geq 1$, then $Q(f)$ can be written as a linear combination of elements of the form $f_{j'}\widetilde Q(f)$ for varying $j'$ and $\ol\deg(\widetilde Q(f))< \ol\deg(Q(f))$. In this case, if $j'\neq r+1$, we have
\begin{align*}
\Phi(e_j\cdot(f_{j'}\widetilde Q(f), I_1))&=\Phi((e_j\otimes k_j)\cdot(f_{j'}\widetilde Q(f),I_1)+(1\otimes e_j)\cdot(f_{j'}\widetilde Q(f),I_1))\\
\text{(because $e_j\cdot u_{I_1}=0$)} &=v^{\delta_{j,r+1}}\Phi([e_j,f_{j'}\widetilde Q(f)], I_1)\\
&=\delta_{j,j'}\Phi\Big(\frac{k_{j}-k_{j}^{-1}}{v-v^{-1}}\widetilde Q(f),I_1\Big) + v^{\delta_{j,r+1}}\Phi(f_{j'}[e_j,\widetilde{Q}(f)],I_1)\\
\text{(by \eqref{eq:Phi-fj})}&=\delta_{j,j'}\frac{k_{j}-k_{j}^{-1}}{v-v^{-1}}\cdot\Phi(\widetilde Q(f),I_1) + v^{\delta_{j,r+1}}f_{j
'}\cdot \Phi([e_j,\widetilde{Q}(f)],I_1)\\
\text{(by induction and because $e_j\cdot u_{I_1}=0$)}&=\delta_{j,j'}\frac{k_{j}-k_{j}^{-1}}{v-v^{-1}}\cdot\Phi(\widetilde Q(f),I_1) + f_{j'}e_j\cdot\Phi(\widetilde{Q}(f),I_1)\\
&=e_jf_{j'}\cdot \Phi(\widetilde{Q}(f),I_1)\\
\text{(by \eqref{eq:Phi-fj})}&=e_j\cdot \Phi(f_{j'}\widetilde{Q}(f),I_1),
\end{align*}
proving the claim. Finally, suppose $j'=r+1$. In that case:
\begin{align*}
\Phi(e_j\cdot(f_{r+1}\widetilde Q(f), I_1))&=v^{\delta_{j,r+1}}\Phi([e_j,f_{r+1}\widetilde Q(f)], I_1)\\
&=\delta_{j,r+1}v\Phi\Big(\frac{k_j-k_j^{-1}}{v-v^{-1}}\widetilde Q(f),I_1\Big) + v^{\delta_{j,r+1}}\Phi(f_{r+1}[e_j,\widetilde{Q}(f)],I_1)\\
\text{(by \eqref{eq:Phi-fr})}&=\delta_{j,r+1}v\Phi\Big(\frac{k_j-k_j^{-1}}{v-v^{-1}}\widetilde Q(f),I_1\Big) \\
&~+ v^{\delta_{j,r+1}+(\mu_1)-(\mu_{r+1}-\delta_{j,r}+\delta_{j,r+1}+1)+1}\ell f_1f_2\cdots f_rf_{r+1}\cdot \Phi([e_j,\widetilde{Q}(f)],I_{r+1})\\
&=\delta_{j,r+1}v\Phi\Big(\frac{k_j-k_j^{-1}}{v-v^{-1}}\widetilde Q(f),I_1\Big) + v^{\mu_1-\mu_{r+1}}\ell f_1f_2\cdots f_rf_{r+1}e_j\cdot\Phi(\widetilde{Q}(f),I_{r+1}),
\end{align*}
where the last equality follows by observing that $k_j\cdot u_{I_{r+1}}=v^{\delta_{j,r}}u_{I_{r+1}}.$ If $j\neq r+1$, the term on the left is zero, whereas the term on the right becomes:
\begin{align*}
v^{\mu_1-\mu_{r+1}}\ell f_1f_2\cdots f_rf_{r+1}e_j\cdot\Phi(\widetilde{Q}(f),I_{r+1})&=v^{\mu_1-\mu_{r+1}}\ell f_1f_2\cdots f_je_j\cdots f_rf_{r+1}\cdot\Phi(\widetilde{Q}(f),I_{r+1})\\
&=v^{\mu_1-\mu_{r+1}}\ell f_1f_2\cdots \left(e_jf_j -\frac{k_j-k_j^{-1}}{v-v^{-1}}\right)\cdots f_rf_{r+1}\cdot\Phi(\widetilde{Q}(f),I_{r+1})\\
&=v^{\mu_1-\mu_{r+1}}\ell f_1f_2\cdots e_jf_j \cdots f_rf_{r+1}\cdot\Phi(\widetilde{Q}(f),I_{r+1})\\
&\text{(because } \ell f_1f_2\cdots f_{j-1}\cdot \Phi(\widetilde{Q}(f),I_{r+1})=0 \text{ by the induction step)}\\
&=v^{\mu_1-\mu_{r+1}}e_j\ell f_1f_2\cdots f_rf_{r+1}\cdot\Phi(\widetilde{Q}(f),I_{r+1})\\
\text{(by \eqref{eq:Phi-fr})}&=e_j\cdot \Phi(f_{r+1} \widetilde{Q}(f),I_{1}).
\end{align*}
Finally, if $j=r+1$, the sum equals:
\begin{align*}
&~v\Phi\Big(\frac{k_{r+1}-k_{r+1}^{-1}}{v-v^{-1}}\widetilde Q(f),I_1\Big) + v^{\mu_1-\mu_{r+1}}\ell f_1f_2\cdots f_rf_{r+1}e_{r+1}\cdot\Phi(\widetilde{Q}(f),I_{r+1})\\
=&~v\Phi\Big(\frac{k_{r+1}-k_{r+1}^{-1}}{v-v^{-1}}\widetilde Q(f),I_1\Big) + v^{\mu_1-\mu_{r+1}}e_{r+1}\ell f_1f_2\cdots f_rf_{r+1}\cdot\Phi(\widetilde{Q}(f),I_{r+1})\\
&-v^{\mu_1-\mu_{r+1}}\ell f_1f_2\cdots f_r \cdot\Phi\Big(\frac{k_{r+1}-k_{r+1}^{-1}}{v-v^{-1}}\widetilde{Q}(f),I_{r+1}\Big)\\
=&~v\Phi\Big(\frac{k_{r+1}-k_{r+1}^{-1}}{v-v^{-1}}\widetilde Q(f),I_1\Big) + v^{\mu_1-\mu_{r+1}}e_{r+1}\ell f_1f_2\cdots f_rf_{r+1}\cdot\Phi(\widetilde{Q}(f),I_{r+1}) -  v\Phi\Big(\frac{k_{r+1}-k_{r+1}^{-1}}{v-v^{-1}}\widetilde{Q}(f),I_{1}\Big)\\
&(\text{where the last equality follows from the fact that }\\
&\ell f_1f_2\cdots f_r \cdot\Phi\Big(\frac{k_{r+1}-k_{r+1}^{-1}}{v-v^{-1}}\widetilde{Q}(f),I_{r+1}\Big) = \Phi\Big(\ell f_1f_2\cdots f_r \cdot\left(\frac{k_{r+1}-k_{r+1}^{-1}}{v-v^{-1}}\widetilde{Q}(f),I_{r+1}\right)\Big)\\
&\text{by the induction step)}\\
=&~v^{\mu_1-\mu_{r+1}}e_{r+1}\ell f_1f_2\cdots f_rf_{r+1}\cdot\Phi(f_{r+1}\widetilde{Q}(f),I_{r+1})\\
\text{(by \eqref{eq:Phi-fr})}=&~e_{r+1}\cdot \Phi(f_{r+1}\widetilde{Q}(f),I_1).
\end{align*}

Lastly, we suppose $z=\ell$. We need to show that $\ell$ acts on $\Phi(Q(f), I_1)$ by $1$. We write $Q(f)$ as a linear combination of elements of the form $f_j\widetilde{Q}(f)$,with $\ol{\deg}(\widetilde{Q}(f))<\ol{\deg}(Q(f))$. If $j\neq r+1$, we have, by \eqref{eq:Phi-fj},
\[\ell\cdot \Phi(f_j\widetilde Q(f),I_1) =\ell f_j\cdot \Phi(\widetilde Q(f),I_1)= \ell f_j\ell\cdot \Phi(\widetilde Q(f),I_1) = f_j\ell\cdot \Phi(\widetilde Q(f),I_1)=f_j\cdot \Phi(\widetilde Q(f),I_1)=\Phi(f_j\widetilde Q(f),I_1)\]
where the second and fourth equality follow by the induction step.
Finally, by \eqref{eq:Phi-fr},
\begin{align*}
\ell\cdot \Phi(f_{r+1}\widetilde Q(f),I_1) &= \ell\cdot(v^{\mu_1-\mu_{r+1}+1}\ell f_1f_2\cdots f_{r+1} \cdot \Phi(\widetilde Q(f),I_{r+1}))\\
&= v^{\mu_1-\mu_{r+1}+1}\ell f_1f_2\cdots f_{r+1} \cdot \Phi(\widetilde Q(f),I_{r+1})\\
&= \Phi(f_{r+1}\widetilde Q(f),I_1),
\end{align*}completing the proof.

Now, we work with general $i$. Recall that, by \eqref{eq:Phi-I+1}, 
\[\Phi( Q(f),I_{i+1})=e_i\cdot \Phi( Q(f), I_i)-v^{-1}\Phi([e_i,Q(f)], I_{i}).\]

We need to show that $\Phi(z\cdot(Q(f), I_{i+1})) = z\cdot \Phi(Q(f), I_{i+1})$ where $z$ is a generator of $MU(n)$. We accomplish this in $3$ steps, dealing with the $f_j$'s, $\ell$ and $e_j$'s separately.

\underline{\textbf{Step $1$: $z=f_j$}}

Let $z=f_j$ and $j\not\in\{i-1,i,r+1\}$. Then,
\begin{align*}\Phi( f_j\cdot(Q(f),I_{i+1}))=\Phi( f_jQ(f),I_{i+1})&=e_i\cdot \Phi( f_jQ(f), I_i)-v^{-1}\Phi([e_i,f_jQ(f)], I_{i})\\
\text{(because $f_j\cdot u_{I_i}=0$ and $[e_i,f_j]=0$) }&=e_i\cdot \Phi(f_j\cdot(Q(f), I_i))-v^{-1}\Phi(f_j[e_i,Q(f)], I_{i})\\
\text{(by the induction step and because $f_j\cdot u_{I_i}=0$) }&=e_if_j\cdot \Phi(Q(f), I_i)-v^{-1}\Phi(f_j\cdot([e_i,Q(f)], I_{i}))\\
\text{(by the induction step and because $[e_i,f_j]=0$) }&=f_je_i\cdot \Phi(Q(f), I_i)-v^{-1}f_j\cdot\Phi([e_i,Q(f)], I_{i}) \\
&=  f_j\cdot\Phi(Q(f),I_{i+1}).\end{align*}
Next, we first consider the case when $i\geq 2$ and $z=f_{i-1}$. Then,
\begin{align*}
f_{i-1}\cdot\Phi( Q(f),I_{i+1})&=e_if_{i-1}\cdot \Phi( Q(f), I_i)-v^{-1}f_{i-1}\cdot\Phi([e_i,Q(f)], I_{i})\\
\text{(by induction hypothesis) }&=e_i\cdot \Phi(f_{i-1}\cdot( Q(f), I_i))-v^{-1}\Phi(f_{i-1}\cdot([e_i,Q(f)], I_{i}))\\
&=e_i\cdot \Phi(f_{i-1}Q(f),I_i) + v^{\mu_{i}-\mu_{i-1}}e_i\cdot \Phi(Q(f), I_{i-1})\\
&~-v^{-1}\Phi(f_{i-1}[e_i,Q(f)], I_{i})-v^{\mu_i-\mu_{i-1}}\Phi([e_i,Q(f)], I_{i-1})\\
\text{(by induction hypothesis and because $e_i\cdot u_{I_{i-1}}=0$) }&=\Phi(f_{i-1}Q(f), I_{i+1})\\
\text{(because $f_{i-1}\cdot u_{I_{i+1}}=0$) } &=\Phi(f_{i-1}\cdot(Q(f), I_{i+1})).
\end{align*}
Next, when $z=f_i$, we get:
\begin{align*}
f_{i}\cdot\Phi( Q(f),I_{i+1})&=f_ie_i\cdot \Phi( Q(f), I_i)-v^{-1}f_{i}\cdot\Phi([e_i,Q(f)], I_{i})\\
\text{(by induction hypothesis and because $f_i\cdot u_{I_i}=0$) }&=e_if_i\cdot \Phi( Q(f), I_i)-\frac{k_i-k_i^{-1}}{v-v^{-1}}\cdot \Phi( Q(f), I_i)-v^{-1}\Phi(f_i[e_i,Q(f)], I_{i})\\
\text{(by induction hypothesis and because $f_i\cdot u_{I_i}=0$) }&=e_i\cdot \Phi(f_iQ(f), I_i)-\frac{k_i-k_i^{-1}}{v-v^{-1}}\cdot \Phi( Q(f), I_i)\\
&~-v^{-1}\Phi([e_i,f_iQ(f)], I_{i})+v^{-1}\Phi\Big(\frac{k_i-k_i^{-1}}{v-v^{-1}}Q(f), I_{i}\Big)\\
\text{(because $\rho(k_i^{\pm 1})=k_i^{\pm 1}\otimes k_i^{\pm 1}$)}&=e_i\cdot \Phi(f_iQ(f),I_i)-v^{-1}\Phi([e_i,f_iQ(f)],I_i)+v^{-1}k_i^{-1}\cdot \Phi(Q(f),I_i)\\
&=\Phi(f_iQ(f),I_{i+1})+ \Phi(k_i^{-1}Q(f),I_i)\\
\text{(because $f_i\cdot u_{I_{i+1}}=I_i$) }&=\Phi(f_i\cdot (Q(f), I_{i+1})).
\end{align*}
Finally, when $z=f_{r+1}$, we see that (Note that $i\leq r$, so $[e_i,f_{r+1}]=0$):
\begin{align*}
f_{r+1}\cdot\Phi( Q(f),I_{i+1})&=e_if_{r+1}\cdot \Phi( Q(f), I_i)-v^{-1}f_{r+1}\cdot\Phi([e_i,Q(f)], I_{i})\\
\text{(by induction hypothesis)}&=e_i\cdot \Phi(f_{r+1}Q(f), I_i)+v^{\mu_{r+2}-\mu_{r+1}}e_i\cdot \Phi( Q(f), (I_i\cup\{r+2\})\sm\{r+1\})\\
&~-v^{-1}\cdot\Phi([e_i,f_{r+1}Q(f)], I_{i})-v^{\mu_{r+2}-\mu_{r+1}-1+\delta_{i,r}}\Phi([e_i,Q(f)], (I_i\cup\{r+2\})\sm\{r+1\})\\
\text{(by Lemma \ref{lem:Phi-ei})}&=\Phi(f_{r+1}Q(f),I_{i+1}) + v^{\mu_{r+2}-\mu_{r+1}} (1-\delta_{i,r})\Phi(Q(f), (I_{i+1}\cup\{r+2\})\sm\{r+1\}))\\
&=\Phi(f_{r+1}\cdot (Q(f), I_{i+1})).
\end{align*}

\underline{\textbf{Step $2$: $z=\ell$}}

We start the induction step with the equality from \eqref{eq:Phi-I+1}:
\[\Phi( Q(f),I_{i+1})=e_i\cdot \Phi( Q(f), I_i)-v^{-1}\Phi([e_i,Q(f)], I_{i}).\]
We first let $i=1$ in the above equality. Then, we have:
\[
\ell\cdot \Phi( Q(f),I_{2}) = \ell e_1\cdot \Phi( Q(f), I_1)-v^{-1}\ell\cdot \Phi([e_1,Q(f)], I_{1})
\]
We induct on $\ol{\deg}(Q(f))$. When $Q(f)=1$, we have by \eqref{eq:Phi1}: 
\begin{align*}\ell\cdot \Phi( 1,I_{2})&=v^{\lambda_1-\lambda_{r+1}}\ell e_1\ell f_1\cdots f_{r-1}f_r\cdot w_{\lambda, r}\\
&=v^{\lambda_1-\lambda_{r+1}}\ell e_1 f_1\cdots f_{r-1}f_r\cdot w_{\lambda, r}\\
&= v^{\lambda_1-\lambda_{r+1}}\ell \left(f_1e_1+\frac{k_1-k_1^{-1}}{v-v^{-1}}\right)f_2\cdots f_{r-1}f_r\cdot w_{\lambda, r},\end{align*}
which is zero since $\ell$ and $e_1$ commute past all the terms $f_2,\ldots, f_r$ and act on $w_{\lambda, r}$ by zero. Next, if $\ol\deg(Q(f))\geq 1$, we can write $Q(f)$ as a linear combination of elements of the form $f_{j}\widetilde Q(f)$ for varying $j$ and $\ol\deg(\widetilde Q(f))< \ol\deg(Q(f)$. When $j\neq 1, r+1$, we have:
\begin{align*}
\ell\cdot \Phi(f_j\widetilde Q(f),I_{2}) &= \ell e_1\cdot \Phi( f_j\widetilde Q(f), I_1)-v^{-1}\ell\cdot \Phi([e_1,f_j\widetilde Q(f)], I_{1})\\
\text{(by \eqref{eq:Phi-fj}) }&=f_j\ell\cdot( e_1\cdot \Phi( \widetilde Q(f), I_1)-v^{-1} \Phi([e_1,\widetilde Q(f)], I_{1}))\\
&=f_j\ell\cdot(\Phi(\widetilde{Q}(f),I_2))\\
\text{(by induction) }&=0.
\end{align*}
Next, suppose $j=r+1$. Then, 
\begin{align*}
\ell\cdot \Phi(f_{r+1}\widetilde Q(f),I_{2}) &= \ell e_1\cdot \Phi( f_{r+1}\widetilde Q(f), I_1)-v^{-1}\ell\cdot \Phi([e_1,f_{r+1}\widetilde Q(f)], I_{1})\\
\text{(by \eqref{eq:Phi-fr}) }&=v^{\mu_{1}-\mu_{r+1}-1}\ell e_1 \ell f_1f_2\cdots f_{r+1}\cdot \Phi( \widetilde Q(f), I_{r+1})-v^{\mu_{1}-\mu_{r+1}-1}\ell\ell f_1f_2\cdots f_{r+1}\cdot \Phi([e_1,\widetilde Q(f)], I_{r+1})\\
&=v^{\mu_{1}-\mu_{r+1}-1}\ell e_1 f_1f_2\cdots f_{r+1}\cdot \Phi( \widetilde Q(f), I_{r+1})-v^{\mu_{1}-\mu_{r+1}-1}\ell f_1f_2\cdots f_{r+1}\cdot \Phi([e_1,\widetilde Q(f)], I_{r+1})\\
&=v^{\mu_{1}-\mu_{r+1}-1}\ell f_1 e_1f_2\cdots f_{r+1}\cdot \Phi( \widetilde Q(f), I_{r+1})-v^{\mu_{1}-\mu_{r+1}-1}\ell f_1f_2\cdots f_{r+1}\cdot \Phi([e_1,\widetilde Q(f)], I_{r+1})\\
&\text{(Here, we use the fact that }\ell\cdot \Phi(\widetilde{Q}(f),I_{r+1})=0 \text{ in order to skip the commutator term}.)\\
&=v^{\mu_{1}-\mu_{r+1}-1}\ell f_1 f_2\cdots f_{r+1}e_1\cdot \Phi( \widetilde Q(f), I_{r+1})-v^{\mu_{1}-\mu_{r+1}-1}\ell f_1f_2\cdots f_{r+1}\cdot \Phi([e_1,\widetilde Q(f)], I_{r+1})\\
\text{(since $e_1\cdot u_{I_{r+1}}=0$)}&=0.
\end{align*}

Lastly, we consider the case when $j=1$. By Corollary~\ref{cor:lastcor} we can assume that $Q(f)$ is of the form $f_1^2Q'(f)$ for some non-commutative monomial $Q'(f)$ (having degree lesser than that of $Q(f)$) or $Q(f)=f_1f_2\cdots f_s$ for some $1\leq s\leq n-1$. In the former case, we have that:
\begin{align*}
\ell\cdot \Phi(f_{1}^2 Q'(f),I_{2}) &= \ell\cdot(\Phi(f_1\cdot(f_1Q'(f),I_2))-v^{\mu_2-\mu_1-2}\Phi(f_1\widetilde{Q}(f),I_1))\\
\text{(by the induction assumption)}&=\ell f_1\cdot\Phi(f_1Q'(f), I_2) - v^{\mu_2-\mu_1-2}\ell\cdot\Phi(f_1Q'(f),I_1)\\
\text{(by the induction assumption)}&= \ell f_1^2\cdot\Phi(Q'(f), I_2) - v^{\mu_2-\mu_1-4}\ell f_1\cdot \Phi(Q'(f),I_1)\\
&- v^{\mu_2-\mu_1-2}\Phi(f_1Q'(f),I_1)\\
\text{(by the induction assumption and $f_1\cdot u_{I_1}=0$)}&= \ell f_1^2\cdot\Phi(Q'(f), I_2) - v^{\mu_2-\mu_1-4}\ell \cdot  \Phi(f_1Q'(f),I_1)\\
&- v^{\mu_2-\mu_1-2}\Phi(f_1Q'(f),I_1)\\
&=\ell f_1^2\cdot\Phi(Q'(f), I_2) - v^{\mu_2-\mu_1-4}(1+v^2)\ell \cdot  \Phi(f_1Q'(f),I_1)\\
\text{(by Equation~\eqref{eq:mun.5})}&=((1+v^2)f_1\ell f_1 - v^2f_1^2\ell)\cdot\Phi(Q'(f), I_2) \\
&- v^{\mu_2-\mu_1-4}(1+v^2)\ell \cdot  \Phi(f_1Q'(f),I_1)\\
\text{(since $\ell\cdot \Phi(Q'(f), I_2)=0$ by induction)}&=(1+v^2)f_1\ell f_1 \cdot\Phi(Q'(f), I_2) - v^{\mu_2-\mu_1-4}(1+v^2)\ell \cdot  \Phi(f_1Q'(f),I_1)\\
&=(1+v^2)f_1\ell \cdot \Phi(f_1Q'(f), I_2) + (1+v^2)v^{\mu_2-\mu_1-4}f_1\ell\cdot \Phi(Q'(f),I_1) \\
&- v^{\mu_2-\mu_1-4}(1+v^2)\ell \cdot  \Phi(f_1Q'(f),I_1)\\
\text{(since $\ell\cdot \Phi(f_1Q'(f), I_2)=0$)}&=0.
\end{align*}

Next, we deal with the case when $Q(f)=f_1f_2\cdots f_s$ for some $s$. If $ s<r+1$, we note that:
\begin{align}
\notag\ell\cdot\Phi(f_1f_2\cdots f_s,I_{2})&=\ell e_1\cdot \Phi(f_1f_2\cdots f_s, I_1)-v^{-1}\ell\cdot \Phi([e_1,f_1f_2\cdots f_s], I_{1})\\
\notag&=\ell e_1\cdot \Phi(f_1f_2\cdots f_s, I_1)-v^{-1}\ell\cdot\Phi\Big(\frac{k_1-k_1^{-1}}{v-v^{-1}}f_2\cdots f_s, I_{1}\Big)\\
\notag&=\ell e_1f_1f_2\cdots f_s\cdot \Phi(1, I_1)-v^{-1}\ell\frac{vk_1-v^{-1}k_1^{-1}}{v-v^{-1}}f_2\cdots f_s\cdot \Phi(1, I_{1})\\
\notag&\text{(by induction assumption)}\\
\notag\text{(by~\eqref{eq:Phi1})}&=v^{\lambda_1-\lambda_{r+1}}\ell e_1f_1f_2\cdots f_s\ell f_1f_2\cdots f_r \cdot  w_{\lambda,r}-v^{\lambda_1-\lambda_{r+1}-1}\ell\frac{vk_1-v^{-1}k_1^{-1}}{v-v^{-1}}f_2\cdots f_s\ell f_1f_2\cdots f_r \cdot  w_{\lambda,r}\\
\notag\text{(by Corollary~\ref{cor:fmove})}&=\frac{v^{\lambda_1-\lambda_{r+1}-1}}{(v+v^{-1})^s}\Big(\ell e_1\ell f_1^2f_2^2\cdots f_s^2 f_{s+1}\cdots f_r -(v+v^{-1})\ell\frac{vk_1-v^{-1}k_1^{-1}}{v-v^{-1}}\ell f_1f_2^2\cdots f_s^2 f_{s+1}\cdots f_r\Big) \cdot  w_{\lambda,r}\\
\notag\text{(by~\eqref{eq:mun.1} and~\eqref{eq:mun.3})}&=\frac{v^{\lambda_1-\lambda_{r+1}-1}}{(v+v^{-1})^s}\ell\Big(e_1 f_1^2 -(v+v^{-1})\frac{vk_1-v^{-1}k_1^{-1}}{v-v^{-1}} f_1\Big)f_2^2\cdots f_s^2 f_{s+1}\cdots f_r \cdot  w_{\lambda,r}.
\end{align}
Evaluating the term inside the middle parentheses in the last line above:
\begin{align*}
 e_1 f_1^2-(v+v^{-1})\frac{vk_1-v^{-1}k_1^{-1}}{v-v^{-1}} f_1& =  f_1e_1f_1+\frac{k_1-k_1^{-1}}{v-v^{-1}}f_1-(v+v^{-1})\frac{vk_1-v^{-1}k_1^{-1}}{v-v^{-1}} f_1\\
& = f_1^2e_1+f_1\frac{k_1-k_1^{-1}}{v-v^{-1}}+ \frac{k_1-k_1^{-1}}{v-v^{-1}}f_1-(v+v^{-1})\frac{vk_1-v^{-1}k_1^{-1}}{v-v^{-1}} f_1\\
& = f_1^2e_1+\frac{v^2k_1-v^{-2}k_1^{-1}}{v-v^{-1}}f_1+\frac{k_1-k_1^{-1}}{v-v^{-1}}f_1-(v+v^{-1})\frac{vk_1-v^{-1}k_1^{-1}}{v-v^{-1}} f_1\\
& = f_1^2e_1.
\end{align*}
Since $e_1$ commutes past all $f_j$'s for $j>1$ to act on $w_{\lambda, r}$ by zero, we conclude that $\ell\cdot \Phi(f_1f_2\cdots f_s, I_2)=0$. Next, we suppose that $s\geq r+1$. In that case, we again have:
\begin{align*}
\ell\cdot\Phi(f_1f_2\cdots f_s,I_{2})&=\ell e_1\cdot \Phi(f_1f_2\cdots f_s, I_1)-v^{-1}\ell\cdot\Phi\Big(\frac{k_1-k_1^{-1}}{v-v^{-1}}f_2\cdots f_s, I_{1}\Big)\\
&=\ell e_1f_1f_2\cdots f_r\cdot \Phi(f_{r+1}f_{r+2}\cdots f_s, I_1)-v^{-1}\ell\frac{vk_1-v^{-1}k_1^{-1}}{v-v^{-1}}f_2\cdots f_r\cdot \Phi(f_{r+1}f_{r+2}\cdots f_s, I_{1})\\
&\text{(by induction assumption)}\\
\text{(by~\eqref{eq:Phi-fr})}&=v^{\lambda_{r+1}-\lambda_1+1}\Big(\ell e_1f_1f_2\cdots f_r-v^{-1}\ell\frac{vk_1-v^{-1}k_1^{-1}}{v-v^{-1}}f_2\cdots f_r\Big)\ell f_1f_2\cdots f_{r+1} \cdot\Phi(f_{r+2}\cdots f_s, I_{r+1})\\
&=v^{\lambda_{r+1}-\lambda_1+1}\Big(\ell e_1f_1f_2\cdots f_r-v^{-1}\ell\frac{vk_1-v^{-1}k_1^{-1}}{v-v^{-1}}f_2\cdots f_r\Big)\ell f_1f_2\cdots f_{r+1}f_{r+2}\cdots f_s \cdot\Phi(1, I_{r+1})\\
&\text{(by induction assumption)}\\
\text{(by~\eqref{eq:topvect})}&=v^{\lambda_{r+1}-\lambda_1+1}\Big(\ell e_1f_1f_2\cdots f_r-v^{-1}\ell\frac{vk_1-v^{-1}k_1^{-1}}{v-v^{-1}}f_2\cdots f_r\Big)\ell f_1f_2\cdots f_{r+1}f_{r+2}\cdots f_s \cdot w_{\lambda, r},
\end{align*}
which is zero by exactly the same argument as in the case when $s<r+1$ using Corollary~\ref{cor:fmove}. This completes the proof of the fact that:
\begin{equation} \label{eq:ell20}\ell\cdot \Phi(Q(f), I_2) = 0 = \Phi(\ell\cdot (Q(f), I_2)).\end{equation}
Finally, we suppose that $i>1$ (so $e_i$ commutes with $\ell$). Then, we have by equation~\eqref{eq:Phi-I+1}:
\[\ell \cdot \Phi(Q(f),I_{i+1})=\ell e_i\cdot \Phi( Q(f), I_i)-v^{-1}\ell\cdot \Phi([e_i,Q(f)], I_{i})=0,\]
where the last equality follows by the induction assumption.

\underline{\textbf{Step $3$: $z=e_j$}}

Let $z=e_j$ where $j\not\in\{i-1,i,i+1\}$. Then, 
\begin{align*}\Phi( e_j\cdot(Q(f),I_{i+1}))=v^{\delta_{j,r+1}}\Phi( e_jQ(f),I_{i+1})&=v^{\delta_{j,r+1}}(e_i\cdot \Phi( e_jQ(f), I_i)-v^{-1}\Phi([e_i,e_jQ(f)], I_{i}))\\
\text{(because $e_j\cdot u_{I_i}=0$ and $[e_i,e_j]=0$) }&=e_i\cdot \Phi(e_j\cdot(Q(f), I_i))-v^{\delta_{j,r+1}-1}\Phi(e_j[e_i,Q(f)], I_{i})\\
\text{(by the induction step and because $e_j\cdot u_{I_i}=0$) }&=e_ie_j\cdot \Phi(Q(f), I_i)-v^{-1}\Phi(e_j\cdot([e_i,Q(f)], I_{i}))\\
\text{(by the induction step and because $[e_i,e_j]=0$) }&=e_je_i\cdot \Phi(Q(f), I_i)-v^{-1}e_j\cdot\Phi([e_i,Q(f)], I_{i}) \\
&=  e_j\cdot\Phi(Q(f),I_{i+1}).\end{align*}

Next, we first consider the case when $i\geq 2$ and $z=e_{i-1}$.
\begin{align*}
&~(v+v^{-1})e_{i-1}\cdot \Phi( Q(f),I_{i+1})\\
=&~(v+v^{-1})e_{i-1}e_i\cdot \Phi( Q(f), I_i)-v^{-1}(v+v^{-1})e_{i-1}\cdot\Phi([e_i,Q(f)], I_{i})\\
=&~(v+v^{-1})e_{i-1}e_i\cdot(e_{i-1}\cdot \Phi( Q(f), I_{i-1})-v^{-1}\Phi([e_{i-1},Q(f)], I_{i-1})) -(1+v^{-2})e_{i-1}\cdot\Phi([e_i,Q(f)], I_{i})\\
=&~(v+v^{-1})e_{i-1}e_ie_{i-1}\cdot \Phi( Q(f), I_{i-1})-(1+v^{-2})e_{i-1}e_i\cdot\Phi([e_{i-1},Q(f)], I_{i-1}) \\
&-(1+v^{-2})e_{i-1}\cdot\Phi([e_i,Q(f)], I_{i})\\
=&~e_{i-1}^2e_i\cdot \Phi( Q(f), I_{i-1})+e_ie_{i-1}^2\cdot \Phi( Q(f), I_{i-1})-(1+v^{-2})e_{i-1}e_i\cdot\Phi([e_{i-1},Q(f)], I_{i-1})\\
&-(1+v^{-2})e_{i-1}\cdot\Phi([e_i,Q(f)], I_{i})\\
=&~e_{i-1}^2\cdot \Phi([e_i,Q(f)], I_{i-1})+e_ie_{i-1}^2\cdot \Phi( Q(f), I_{i-1})-(1+v^{-2})e_{i-1}\cdot\Phi([e_i,[e_{i-1},Q(f)]], I_{i-1})\\
&-(v+v^{-1})\Phi([e_{i-1},[e_i,Q(f)]], I_{i})\\
=&~v^{-2}\Phi([e_{i-1}^2,[e_i,Q(f)]], I_{i-1})+(v+v^{-1})\Phi([e_{i-1},[e_i,Q(f)]], I_i)+v^{-2}e_i\cdot \Phi([e_{i-1}^2,Q(f)], I_{i-1})\\
&+(v+v^{-1})e_i\cdot \Phi([e_{i-1},Q(f)], I_{i})-(1+v^{-2})e_{i-1}\cdot\Phi([e_i,[e_{i-1},Q(f)]], I_{i-1}) -(v+v^{-1})\Phi([e_{i-1},[e_i,Q(f)]], I_{i})\\
=&~v^{-2}\Phi([e_{i-1}^2,[e_i,Q(f)]], I_{i-1})+(v+v^{-1})\Phi([e_{i-1},[e_i,Q(f)]], I_i)+v^{-2} \Phi([e_i,[e_{i-1}^2,Q(f)]], I_{i-1})\\
&+(1+v^{-2})\Phi([e_i,[e_{i-1},Q(f)]], I_{i})+(v+v^{-1})\Phi([e_{i-1},Q(f)], I_{i+1})-v^{-1}(1+v^{-2})\Phi([e_{i-1},[e_i,[e_{i-1},Q(f)]]], I_{i-1})\\
&-(1+v^{-2})\Phi([e_i,[e_{i-1},Q(f)]], I_{i})-(v+v^{-1})\Phi([e_{i-1},[e_i,Q(f)]], I_{i})\\
=&~(v+v^{-1})\Phi([e_{i-1},Q(f)], I_{i+1})\\
=&~(v+v^{-1})\Phi(e_{i-1}\cdot (Q(f), I_{i+1})).
\end{align*}
The case when $z=e_i$ is similar and has been omitted. (When $i=1$, the computation uses relation~\eqref{eq:mun.4} instead of~\eqref{eq:usln.5} as well as equation~\eqref{eq:ell20}.) 

Finally, when $z=e_{i+1}$ and $i\leq r-1$, we have the equality:
\[e_{i+1}\cdot\Phi( Q(f),I_{i+1}) = \Phi(Q(f),I_{i+2})+v^{-1}\Phi([e_{i+1},Q(f)], I_{i+1}),\]
by \eqref{eq:Phi-I+1} with index $i+1$ instead of $i$, proving the required claim. When $i=r$, notice that $e_{r+1}\cdot (Q(f), I_{r+1}) = ([e_{r+1},Q(f)], I_{r+1})$, so we need to prove the equality:
\begin{equation}\label{eq:er+1}e_{r+1}\cdot \Phi(Q(f), I_{r+1}) = \Phi([e_{r+1},Q(f)], I_{r+1}).\end{equation}
When $Q(f)=1$, the RHS of \eqref{eq:er+1} is zero, whereas the LHS is equal to:
\[e_{r+1}\cdot \Phi(1, I_{r+1})= e_{r+1}\cdot w_{\lambda,r}=0,\]
where the first equality follows from~\eqref{eq:topvect} and the last equality from the fact that $w_{\lambda,r}$ is a highest weight vector. 
Next, we suppose that $\ol\deg(Q(f))\geq 1$. We can write $Q(f)$ as a linear combination of elements of the form $f_{j'}\widetilde Q(f)$ for varying $j'$ and $\ol\deg(\widetilde Q(f))< \ol\deg(Q(f))$. Replacing $Q(f)$ by $f_{j'}\widetilde Q(f)$ in \eqref{eq:er+1}, it is clear that the equality follows by the induction assumption as long as $j'\neq r,r+1$. Suppose $j'=r$. Then,
\begin{align*}
e_{r+1}\cdot \Phi(f_{r}\widetilde Q(f), I_{r+1})&=e_{r+1}f_{r}\cdot \Phi(\widetilde Q(f), I_{r+1}) - v^{\mu_{r+1}-\mu_r-2}e_{r+1}\cdot\Phi(\widetilde{Q}(f),I_r)\\
\text{(because $[e_{r+1},f_r]=0$ and by induction) }&=f_{r}e_{r+1}\cdot \Phi(\widetilde Q(f), I_{r+1}) - v^{\mu_{r+1}-\mu_r-1}\Phi([e_{r+1,}\widetilde{Q}(f)],I_r)\\
\text{(by induction) }&=f_{r}\cdot \Phi([e_{r+1},\widetilde Q(f)], I_{r+1}) - v^{\mu_{r+1}-\mu_r-1}\Phi([e_{r+1,}\widetilde{Q}(f)],I_r)\\
\text{(by induction and by the action of $f_r$) }&=\Phi(f_r[e_{r+1},\widetilde Q(f)], I_{r+1})\\
\text{(because $[e_{r+1},f_r]=0$) }&=\Phi([e_{r+1},f_r\widetilde Q(f)], I_{r+1}),
\end{align*}
which is the required equality. Finally, when $j'=r+1$, we have:
\begin{align*}
e_{r+1}\cdot \Phi(f_{r+1}\widetilde Q(f), I_{r+1})&=e_{r+1}f_{r+1}\cdot \Phi(\widetilde Q(f), I_{r+1})\\
&=f_{r+1}e_{r+1}\cdot \Phi(\widetilde Q(f), I_{r+1}) + \frac{k_{r+1}-k_{r+1}^{-1}}{v-v^{-1}}\cdot \Phi(\widetilde Q(f), I_{r+1})\\
\text{(by induction) }&=f_{r+1}\cdot \Phi([e_{r+1},\widetilde Q(f)], I_{r+1}) + \frac{k_{r+1}-k_{r+1}^{-1}}{v-v^{-1}}\cdot \Phi(\widetilde Q(f), I_{r+1})\\
\text{(by induction and by the actions of $f_{r+1}$ and $k_{r+1}$) }&=\Phi(f_{r+1}[e_{r+1},\widetilde Q(f)], I_{r+1}) + \Phi\Big(\frac{k_{r+1}-k_{r+1}^{-1}}{v-v^{-1}}\widetilde Q(f), I_{r+1}\Big)\\
&=\Phi([e_{r+1},f_{r+1}\widetilde Q(f)], I_{r+1}).
\end{align*}
This proves equivariance of the map $\Phi$ with respect to all the generators, thus proving the lemma.
\end{proof}

\begin{lemma} \label{lem:general}
For any $I\in [n]_r$, $Q(f)\in \bb$ and $z$ a generator of $MU(n)$, we have the equality:
\[\Phi(z\cdot (Q(f),I)) = z\cdot \Phi(Q(f), I).\]
\end{lemma}

\begin{proof}
Let $I=\{i_1<i_2<\dots<i_r\}$. We use induction on the lexicographic order of $I$. The case when $i_r\leq r+1$ corresponds to $I=I_i$ for some $i$, and has been dealt with in 
Lemma \ref{lem:I_r+1}. Pick the largest $j$ such that $i_j\in I$, but $i_j-1\not\in I$. Let $I'=(I\cup\{i_j-1\})\sm\{i_j\}$. Then, by~\eqref{eq:Phi-I'},
\[\Phi(Q(f),I):=v^{\mu_{i_j-1}-\mu_{i_j}}f_{i_j-1}\cdot \Phi(Q(f),I')-v^{\mu_{i_j-1}-\mu_{i_j}}\Phi(f_{i_j-1}Q(f),I').\]
 As long as $z\not \in\{f_{i_j}, f_{i_j-1}, f_{i_j-2}, e_{i_j-1}\}$, the claim of the lemma follows directly by acting on both sides of the above equation by $z$: because $z f_{i_j-1}=f_{i_j-1}z$, we can use the induction hypothesis since $I'<I$ in the lexicographic order. (Note that we necessarily have $i_j-1>1$ and so $\ell$ commutes with $f_{i_j-1}$.)

 
 When $z=e_{i_j-1}$, we see that:
 \begin{align*}
 v^{\mu_{i_j}-\mu_{i_j-1}}e_{i_j-1}\cdot \Phi(Q(f),I)&=e_{i_j-1}\cdot(f_{i_j-1}\cdot \Phi(Q(f),I')-\Phi(f_{i_j-1}Q(f),I'))\\
 &=e_{i_j-1}f_{i_j-1}\cdot \Phi(Q(f),I')-e_{i_j-1}\cdot\Phi(f_{i_j-1}Q(f),I')\\
\text{(by induction and because $e_{i_j-1}u_{I'}=0$)}&=f_{i_j-1}e_{i_j-1}\cdot \Phi(Q(f),I')+\frac{k_{i_j-1}-k_{i_j-1}^{-1}}{v-v^{-1}}\cdot \Phi(Q(f),I')\\&~-v\Phi([e_{i_j-1},f_{i_j-1}Q(f)],I')\\
 \text{(by induction and because $e_{i_j-1}u_{I'}=0$)}  &=vf_{i_j-1}\cdot \Phi([e_{i_j-1},Q(f)],I')+\frac{k_{i_j-1}-k_{i_j-1}^{-1}}{v-v^{-1}} \cdot\Phi(Q(f),I')\\&~-v\Phi(f_{i_j-1}[e_{i_j-1},Q(f)],I')-v\Phi\Big(\frac{k_{i_j-1}-k_{i_j-1}^{-1}}{v-v^{-1}}Q(f),I'\Big)\\
\text{(by \eqref{eq:Phi-I'} and by the action of $k^{\pm 1}_{i_j-1}$)}  &=v^{\mu_{i_j}-\mu_{i_j-1}-1}\Phi([e_{i_j-1},Q(f)],I)+\Phi(k_{i_j-1}^{-1}Q(f),I')\\
    &=v^{\mu_{i_j}-\mu_{i_j-1}-1}\Phi([e_{i_j-1},Q(f)],I)+v^{\mu_{i_j}-\mu_{i_j-1}}\Phi(Q(f),I')\\
\text{(by the action of $e_{i_j-1}$)}  &=v^{\mu_{i_j}-\mu_{i_j-1}}\Phi(e_{i_j-1}\cdot(Q(f),I)).
 \end{align*}

Next, we consider the action of the elements $\{f_{i_j}, f_{i_j-1}, f_{i_j-2}\}$. We first consider the case when $j\neq 1$ and $i_{j-1}=i_j-2$. (Note that $i_{j-1}< i_j-1$ since $i_j-1\not\in I$.) In that case, we have by \eqref{eq:Phi-I'} that:
\begin{align*}f_{i_j-2}\cdot \Phi(Q(f),I) &= \Phi(f_{i_j-2}Q(f), I) + v^{\mu_{i_j-1}-\mu_{i_j-2}}\Phi(Q(f), (I\cup\{i_j-1\})\sm\{i_j-2\})\\
&= \Phi(f_{i_j-2}\cdot (Q(f),I)),\end{align*}
proving equivariance with respect to $f_{i_j-2}$. Also, suppose it is the case that $\{i_1,i_2,\dots, i_{j-1}\} \neq \{1,2,\dots, j-1\}$. In that case, we can choose the minimum $k$ with $1\leq k\leq j-1$ such that $i_k\in \{i_1,i_2,\dots, i_{j-1}\}$ but $i_k-1\not \in \{i_1,i_2,\dots, i_{j-1}\}$. Let $\widetilde I=(I\cup\{i_k-1\})\sm \{i_k\}$. Then, since $\widetilde{I}<I$, by induction we have for all $Q(f)$ the equality:
\[\Phi(Q(f),I)=v^{\mu_{i_k-1}-\mu_{i_k}}f_{i_k-1}\cdot \Phi(Q(f),\widetilde I)-v^{\mu_{i_k-1}-\mu_{i_k}}\Phi(f_{i_k-1}Q(f),\widetilde I).\]
Note that $i_k\leq i_{j-1}= i_j-2$. Then, we must have that $[f_{i_k-1},f_i]=0$ for $i\in\{i_j, i_j-1\}$. In that case, we can apply $f_i$ to both sides of the above equality to conclude the proof. On the other hand, if $\{i_1,i_2,\dots, i_{j-1}\}=\{1,2,\dots,j-1\}$ and $i_{j-1}=i_j-2$, this would imply that $I=I_{j}$, which is a case that has already been dealt with in the previous lemma.

Therefore, we can assume that $i_{j-1}<i_j-2$ or $j=1$. In that case, we define $I''=(I'\cup \{i_j-2\})\sm\{i_j-1\}$. So, by induction, we have the equality:
\[\Phi(Q(f),I')=v^{\mu_{i_j-2}-\mu_{i_j-1}}f_{i_j-2}\cdot \Phi(Q(f),I'')-v^{\mu_{i_j-2}-\mu_{i_j-1}}\Phi(f_{i_j-2}Q(f),I'').\]
With that, we can consider the case when $z=f_{i_j-2}$:
\begin{align*}
&~(v+v^{-1})v^{\mu_{i_j}-\mu_{i_j-2}}f_{i_j-2}\cdot \Phi(Q(f), I)\\
=&~(v+v^{-1})v^{\mu_{i_j-1}-\mu_{i_j-2}}f_{i_j-2}f_{i_j-1}\cdot \Phi(Q(f),I')-(v+v^{-1})v^{\mu_{i_j-1}-\mu_{i_j-2}}f_{i_j-2}\cdot\Phi(f_{i_j-1}Q(f),I')\\
=&~(v+v^{-1})f_{i_j-2}f_{i_j-1}\cdot(f_{i_j-2}\cdot \Phi(Q(f),I'')-\Phi(f_{i_j-2}Q(f),I''))-(v+v^{-1})v^{\mu_{i_j-1}-\mu_{i_j-2}}f_{i_j-2}\cdot\Phi(f_{i_j-1}Q(f),I')\\
=&~(v+v^{-1})f_{i_j-2}f_{i_j-1}f_{i_j-2}\cdot \Phi(Q(f),I'')-(v+v^{-1})f_{i_j-2}f_{i_j-1}\cdot \Phi(f_{i_j-2}Q(f),I'')\\
&-(v+v^{-1})v^{\mu_{i_j-1}-\mu_{i_j-2}}f_{i_j-2}\cdot\Phi(f_{i_j-1}Q(f),I')\\
=&~f_{i_j-2}^2f_{i_j-1}\cdot\Phi(Q(f),I'') + f_{i_j-1}f_{i_j-2}^2\cdot\Phi(Q(f),I'')\\
&-(v+v^{-1})f_{i_j-2}f_{i_j-1}\cdot \Phi(f_{i_j-2}Q(f),I'')-(v+v^{-1})v^{\mu_{i_j-1}-\mu_{i_j-2}}f_{i_j-2}\cdot\Phi(f_{i_j-1}Q(f),I')\\
=&~f_{i_j-2}^2\cdot\Phi(f_{i_j-1}Q(f),I'')+ f_{i_j-1}f_{i_j-2}^2\cdot\Phi(Q(f),I'')\\ 
&-(v+v^{-1})f_{i_j-2}\cdot \Phi(f_{i_j-1}f_{i_j-2}Q(f),I'')-(v+v^{-1})v^{\mu_{i_j-1}-\mu_{i_j-2}}\Phi(f_{i_j-2}f_{i_j-1}Q(f),I')\\
=&~\Phi(f_{i_j-2}^2f_{i_j-1}Q(f),I'')+ (1+v^{2})v^{\mu_{i_j-1}-\mu_{i_j-2}-1}\Phi(f_{i_j-2}f_{i_j-1}Q(f), I')+ f_{i_j-1}\cdot\Phi(f_{i_j-2}^2Q(f),I'')\\ 
&+(1+v^{2})v^{\mu_{i_j-1}-\mu_{i_j-2}}f_{i_j-1}\cdot \Phi(f_{i_j-2}Q(f), I')-(v+v^{-1})f_{i_j-2}\cdot \Phi(f_{i_j-1}f_{i_j-2}Q(f),I'')\\
&-(v+v^{-1})v^{\mu_{i_j-1}-\mu_{i_j-2}}\Phi(f_{i_j-2}f_{i_j-1}Q(f),I')\\
=&~\Phi(f_{i_j-2}^2f_{i_j-1}Q(f),I'')+ (1+v^{2})v^{\mu_{i_j-1}-\mu_{i_j-2}-1}\Phi(f_{i_j-2}f_{i_j-1}Q(f), I')+ \Phi(f_{i_j-1}f_{i_j-2}^2Q(f),I'')\\ 
&+(1+v^{2})v^{\mu_{i_j-1}-\mu_{i_j-2}}\Phi(f_{i_j-1}f_{i_j-2}Q(f), I')+(1+v^{2})v^{\mu_{i_j}-\mu_{i_j-2}-1} \Phi(f_{i_j-2}Q(f), I)\\
&-(v+v^{-1})\Phi(f_{i_j-2}f_{i_j-1}f_{i_j-2}Q(f),I'')-(v+v^{-1})v^{\mu_{i_j-1}-\mu_{i_j-2}+1}\Phi(f_{i_j-1}f_{i_j-2}Q(f),I')\\
&-(v+v^{-1})v^{\mu_{i_j-1}-\mu_{i_j-2}}\Phi(f_{i_j-2}f_{i_j-1}Q(f),I')\\
=&~(v+v^{-1})v^{\mu_{i_j}-\mu_{i_j-2}} \Phi(f_{i_j-2}Q(f), I)\\
=&~(v+v^{-1})v^{\mu_{i_j}-\mu_{i_j-2}} \Phi(f_{i_j-2}\cdot(Q(f), I)).
\end{align*}
The case when $z=f_{i_j-1}$ is fairly similar and has been omitted. The last case that remains to be dealt with is when $z=f_{i_j}$. For this, we first consider the case when $j\neq r$. In that case, $i_{j+1}=i_j+1$. Define $\widetilde{I'} = (I'\cup\{i_j\})\sm\{i_j+1\}$. Then, by induction, we have the equality:
\[\Phi(Q(f),I')=v^{\mu_{i_j}-\mu_{i_j+1}}f_{i_j}\cdot \Phi(Q(f),\widetilde{I'})-v^{\mu_{i_j}-\mu_{i_j+1}}\Phi(f_{i_j}Q(f),\widetilde{I'}).\]
Then, we can check equivariance with respect to $z=f_{i_j}$ by a computation very similar to the one above using this equality and the fact that $\widetilde{I'}<I'<I$.

We finally are left with the case when $j=r\leq n-1$. In that case, we have by~\eqref{eq:Phi-I'}:
\[f_{i_j}\cdot \Phi(Q(f),I) = \Phi(f_{i_j}Q(f),I) + v^{\mu_{i_j+1}-\mu_{i_j}}\Phi(Q(f),(I\cup\{i_j+1\})\sm\{i_j\}),\]
which completes the proof of the claim.
\end{proof}

\section{Classification of finite dimensional representations}\label{sec:ss}

In this section, we prove a complete classification of finite dimensional representations of $MU(n)$ of Type $\mathbf{1}$, by showing 
that the 
representations $\{L_{\lambda,r}\}_{\lambda,r}$ constructed in Section~\ref{sec:mir-reps} provide a complete list of isomorphism classes of simple objects in the category of finite dimensional representations, and that the category is semisimple.

\begin{theorem} \label{thm:class}
Let $M$ be a finite dimensional simple $MU(n)$-module of depth $r$ for some $0\leq r\leq n$. Then $M\simeq L_{\lambda,r}$ for some dominant integral weight $\lambda$.
\end{theorem}
\begin{proof}
Let $x\in M$ be a mirabolic highest weight vector of depth $r$ and fix a weight $\mu=(\mu_1,\mu_2,\cdots,\mu_{n})$ such that $k_i\cdot x = v^{\mu_i-\mu_{i+1}}x.$ Since $M$ is finite dimensional, we must have that $\mu$ is a dominant integral weight for $\uvsln$. Define:
\[\lambda:=\mu - \omega_r.\]
If $r\neq0,n$, notice that since $f_r\cdot x\neq 0$, by \cite[Lemma 5.4(b)]{J} we have $(\mu,\alpha_r)\geq 1$. Hence $(\mu-\omega_r,\alpha_r)=(\mu,\alpha_r)-1\geq 0$, and for $i\neq r$, $(\mu-\omega_r,\alpha_i)=(\mu,\alpha_i)\geq 0$. Thus, $\lambda=\mu-\omega_r$ is again a dominant integral weight.

By the construction of the representation $V_{\lambda,r}$ and Lemma~\ref{lem:maincalc}, there must exist a surjection $V_{\lambda,r}\to M$ of $MU(n)$-modules that sends the vector $w_{\lambda, r}$ to $x$. But then, Theorem~\ref{prop:Verma} implies that $V_{\lambda,r}\simeq M_{\lambda, r}$ and the representation $M_{\lambda,r}$ has a unique finite dimensional quotient $L_{\lambda, r}$ by Proposition~\ref{prop:unique}. So, we conclude that $M\cong L_{\lambda,r}$, completing the proof.
\end{proof}

Next, we show that the category of finite dimensional $MU(n)$-modules is semisimple, which 
is equivalent to showing that there are no non-trivial extensions 
between any of the irreducible representations. We can again restrict our attention to irreducible representations of Type $\bold{1}$.

From Remark \ref{rem:anti-inv}, we recall the anti-involution $\star$ on $MU(n)$:
\[\ell^\star=\ell,\quad (k_i^{\pm 1})^\star=k_i^{\pm 1},\quad e_i^\star=f_i\quad f_i^\star=e_i.\]
Using this anti-involution, we get that for any finite dimensional $MU(n)$-module $V$, there is a natural $MU(n)$-action on the dual space $V^*$ given by $(z\cdot \varphi)(x)=\varphi(z^\star\cdot x)$ for all $z\in MU(n)$, $x\in V$, $\varphi\in V^*$.

\begin{lemma}
If $V$ is a finite dimensional simple $MU(n)$-representation, then $V\cong V^*$.
\end{lemma}

\begin{proof}
Note that if $V$ is simple, so is $V^*$. Also, as `$\star$' acts via the identity on the $k_i$'s and $\ell$, we note that the modules $V$ and $V^*$ have the same mirabolic weight space decomposition. Hence, to prove the lemma, it suffices to show that if $(\lambda, r)\neq (\mu,s)$, then the simple modules $L_{\lambda,r}$ and $L_{\mu,s}$ do not have the same mirabolic weight space decomposition. This follows from the proof of Proposition~\ref{prop:nonisom}.
\end{proof}

\begin{theorem} \label{thm:ss}
The category of finite dimensional $MU(n)$-representations is semisimple.
\end{theorem}

\begin{proof}
(This proof adapts the arguments of Theorem 5.17 of \cite{J}.) Pick any two simple representations $L_{\lambda,r}$ and $L_{\mu,s}$. Suppose we have a short exact sequence of $MU(n)$-representations:
\[0\ra L_{\mu,s} \xra{i} V \xra{\pi} L_{\lambda,r}\ra 0.\]
Proving the theorem is equivalent to showing that the above sequence splits. Let $x_{\lambda,r} \in L_{\lambda,r}$ be the mirabolic highest weight vector $x_{\lambda}\otimes u$ having weight $\lambda+\omega_r$ and depth $r$.

Suppose $\lambda+\omega_r \not\leq \mu + \omega_s $. Pick $x' \in V$ such that $\pi(x')=x_{\lambda,r}$ and having the maximal possible depth amongst all such vectors. Then, $x'$ has weight $\lambda+\omega_r$ and depth at most $r$. Furthemore, the assumption $\lambda+\omega_r \not\leq \mu + \omega_s $ ensures that $\lambda+\omega_r$ is not a weight of $L_{\mu,s}$ and that $x'$ is a highest weight vector of $V$. If the depth of $x'$ is $t$ and $t<r$, by Corollary~\ref{cor:depth-increase}, there exists $c\in K$ such that the vector:
\[\ol x:= x' - ce_t e_{t-1}\cdots e_2e_1\ell f_1f_2\cdots f_{t-1}f_t\cdot x'\]
has depth $>t$. But then, we can compute:
\[
\pi(\ol x) = \pi(x') - ce_t e_{t-1}\cdots e_2e_1\ell f_1f_2\cdots f_{t-1}f_t\cdot \pi(x') = \pi(x') -0= x_{\lambda,r},
\]
where we have used the fact that $\pi(x')=x_{\lambda,r}$ has depth $r>t$ for the second equality. Therefore, $\ol x$ is a vector of depth $>t$ such that $\pi(\ol x)=x_{\lambda, r}$, contradicting the maximality of the depth of $x'$. Hence, we must have that $x'$ has depth exactly $r$. Finally, if $x'$ is not a mirabolic weight vector (which can only happen if $r=0$), we replace it by $\ell\cdot x'$.

Thus, we have that $x'$ is a mirabolic highest weight vector of depth $r$ and weight $\lambda+\omega_r$, that has maximal depth in its weight space (since there are no other vectors with the same weight). Then, by Lemma~\ref{lem:maincalc} and the construction of the representation $V_{\lambda,r}$, the representation $V'=MU(n) \cdot x'\sub V$ must be a (finite dimensional) quotient of $V_{\lambda, r}$. As a consequence, by the isomorphism $V_{\lambda,r}\simeq M_{\lambda,r}$ and Proposition~\ref{prop:unique}, we conclude that $V'\simeq L_{\lambda, r}$. Therefore, since $\pi(V')=L_{\lambda, r}$, we must have $\ker(\pi) \cap V' = 0$, and so, $V=V'\oplus \ker(\pi)$, giving us the required splitting.

Similarly, if we have that $\lambda+\omega_r<\mu+\omega_s$, we can dualize the above short exact sequence to get:
\[0\ra L_{\lambda,r}^* \xra{\pi^*} V^* \xra{i^*} L_{\mu,s}^*\ra 0.\]
We now use the previous lemma and repeat the argument above to get that this sequence splits, which implies that the original sequence also splits.

Finally, we suppose that $\lambda+\omega_r=\mu+\omega_s$ and $r\geq s$. (If $r< s$, we can dualize the short exact sequence and proceed as above.) In that case, the weight space in $V$ with weight $\lambda+\omega_r$ is two dimensional. However, we can still repeat the argument in the first half of the proof since the vector $x'$ we constructed is a mirabolic highest weight vector of depth $r$, and thus has the maximal depth in its weight space. Thus, we can conclude that the sequence splits in this case too, completing the proof.
\end{proof}

\begin{proposition} \label{prop:weightspace}
Let $V$ be a finite dimensional $MU(n)$-module, then its decomposition as a direct sum of simples is uniquely determined by the dimensions of its mirabolic weight spaces $V_{\mu,\epsilon}$. 
\end{proposition}
\begin{proof}
Let $\lambda$ be a maximal weight of $V$ as a $U_v(\sll_n)$-module (i.e. as weight spaces for the $U_v(\sll_n)$-module structure, we have $V_\lambda\neq 0$ and for all $\mu$ such that $V_\mu\neq 0$ we have $\mu\not >\lambda$). Then, a decomposition of $V$ into simple $MU(n)$-modules must contain copies of the modules $L_{\lambda-\omega_r,r}$ for all $0\leq r\leq n$ such that $\lambda-\omega_r$ is dominant, since those are exactly the simples with highest weight $\lambda$.
Let $V':=\bigoplus_{r=0}^n (L_{\lambda-\omega_r,r})^{\oplus s_r}\subset V$ be the sum of all the copies of these modules in the decomposition of $V$.

We consider the extremal weight spaces for the $U_v(\sll_n)$-action, which are $V_{w\cdot\lambda}$, with $w\in \mathcal{S}_n$ (the Weyl group). Notice that, since $L_{\lambda-\omega_r,r}$ has a one dimensional highest weight space, spanned by the vector $x_{\lambda-\omega_r}\otimes u_{\{1,\ldots,r\}}$, it will also have one dimensional extremal weight spaces such that, for all $w\in \mathcal{S}_n$, they are spanned by $x_{w\cdot(\lambda-\omega_r)}\otimes u_{\{w(1),\ldots,w(r)\}}$. Here $x_{w\cdot(\lambda-\omega_r)}$, for varying $w$, span the extremal weights spaces (which are one dimensional) for the simple $\uvsln$-module $L_{\lambda-\omega_r}$. We then have the action of $\ell$ on the extremal weight spaces given by 
$$\ell\cdot x_{w\cdot(\lambda-\omega_r)}\otimes u_{\{w(1),\ldots,w(r)\}}=\epsilon x_{w\cdot(\lambda-\omega_r)}\otimes u_{\{w(1),\ldots,w(r)\}}$$
with
$$\epsilon=\begin{cases}0 & \text{ if }1\in\{w(1),\ldots,w(r)\},\\ 1 & \text{ if }1\not\in\{w(1),\ldots,w(r)\}.\end{cases}$$ 
We now fix the element of the Weyl group to be the $n$-cycle 
$$w:=(n (n-1) (n-2) \ldots 2 1)\in \mathcal{S}_n.$$ Then, the weight space with weight $w^i\cdot \lambda$ in $L_{\lambda-\omega_r,r}$ is such that the eigenvalue for $\ell$ on it is $0$ if $0\leq i\leq r-1$, and $1$ if $r\leq i\leq n-1$. 

We define $d_{i,\epsilon}:=\dim(V_{w^i\cdot\lambda,\epsilon})=\dim(V'_{w^i\cdot\lambda,\epsilon})$, where the equality follows from the fact that $\lambda$ was assumed to be maximal, hence those weight spaces cannot appear in other simple summands in the decomposition of $V$. 
Since $V':=\bigoplus_{r=0}^n (L_{\lambda-\omega_r,r})^{\oplus s_r}$, this gives us a system of equations
$$d_{i,1}=s_0+\ldots+s_i,\qquad d_{i,0}=s_{i+1}+\ldots s_{n},$$
which, given the $d_{i,\epsilon}$'s, uniquely determines the $s_r$'s by
$$s_0=d_{0,1},\quad s_n=d_{n-1,0},\quad s_r=d_{r,1}-d_{r-1,1},\quad 1\leq r\leq n-1.$$
We then conclude the proof by using semisimplicity and iterating this procedure after removing from $V$ all the weight spaces corresponding to the simple modules $L_{\lambda-\omega_r,r}$ that we have already found and choosing a maximal weight among the remaining nonzero weight spaces of $V$ until there are none left.
\end{proof}

\subsection{Braided module category}
Let $\cc$ and $\mm$ be the categories of finite dimensional representations of $\uvsln$ and $MU(n)$ respectively. In this section, we show that $\mm$ is a \emph{braided module category} over $\cc$. We refer the reader to \cite{EGNO} for definitions and background on tensor categories.

It is well-known (see, for example, \cite[$\mathsection{8.3}$]{EGNO}) that $\cc$ is a \emph{braided monoidal category} where the monoidal structure is given by tensor product of representations and the braiding $c_{X,Y}:X\otimes Y\to Y\otimes X$ for all $X,Y\in \mathrm{Ob}(\cc)$ is given by a universal $R$-matrix, which is an invertible element in a certain completion of ${\uvsln\otimes \uvsln}$. The $R$-matrix satisfies the \emph{quantum Yang-Baxter equation}, which is equivalent to the braid identities satisfied by $c_{X,Y}$.

Furthermore, we note that $\cc$ is a braided $\cc$-module category over itself. The action bifunctor $\rhd$ is again given by the tensor product of representations and the braiding $e_{X,M}:X\rhd M\to X\rhd M$ for all $X,M\in \mathrm{Ob}(\cc)$ is defined as the composition $c_{M,X}\circ c_{X,M}$.

Recall the comodule map $\rho:MU(n)\to \uvsln\otimes MU(n)$. This map allows us to view the category $\mm$ as a $\cc$-module category: We define the action bifunctor via
\begin{align*}
\rhd:\cc\times\mm&\to \mm\\
X\rhd M&:=X\otimes M
\end{align*}
for all $X\in \cc$ and $M\in \mm$, where the $MU(n)$-action on $X\otimes M$ is via the map $\rho$.

By semisimplicity, given any $MU(n)$-module $M$, there exists a canonical decomposition:
\[M=\bigoplus_{r=0}^n M_r,\]
where, for any $0\leq r\leq n$, the module $M_r$ is equal to the span of the isotypic components of the simple representations $L_{\lambda,r}$ in $M$ having depth $r$. We define $\mm_r$ to be the full subcategory of $\mm$ whose objects are isomorphic to direct sums of simple $MU(n)$-modules of depth $r$.

The following lemma is an immediate consequence of definitions.

\begin{lemma}
For any $r$, the category $\mm_r$ is a $\cc$-module subcategory of $\mm$. Furthermore, the functor
\[\bigoplus_{r=0}^n \mm_r \to \mm\]
given by taking direct sums of objects and morphisms is an equivalence of $\cc$-module categories, whose inverse is given by sending an object $M$ to the $(r+1)$-tuple $(M_0, M_1, \dots, M_n)$ given by the decomposition above.
\end{lemma}

Thus, in order to construct a braided $\cc$-module structure over $\mm$, it suffices to construct such a structure over each $\mm_r$. This is accomplished with the aid of the following lemma:

\begin{lemma}
Fix $0\leq r\leq n$. The assignment
\begin{align*}
F_r:\cc&\to \mm_r\\
X&\mapsto X\rhd W_r
\end{align*}
induces an equivalence of $\cc$-module categories.
\end{lemma}

\begin{proof}
We describe the construction of an inverse $G_r$ to $F_r$. Given any $M\in \mm_r$, we can canonically decompose $M$ into a direct sum of isotypic components for simple modules of depth $r$. So, without loss of generality, we will assume that $M$ is equal to one such isotypic component corresponding to the simple $L_{\lambda,r}$ for a fixed dominant integral weight $\lambda$.

Let $M_{hw}\sub M$ denote the weight space of $M$ corresponding to the weight $\lambda+\omega_r$. Thus, we have a canonical $MU(n)$-module isomorphism $L_{\lambda,r}\otimes_{\CC} M_{hw} \xra{\sim} M$. Then, the inverse $G_r$ to $F_r$ is given via the formula $G_r(M):=L_{\lambda}\otimes_{\CC} M_{hw}$. That $G_r$ is an inverse to $F_r$ follows by diagram chasing.
\end{proof}

As a consequence of the above lemmas, we see that as a $\cc$-module category, the category $\mm$ is equivalent to the \emph{free module category} $\cc^{\oplus(n+1)}$. Hence, we can transport the braiding on $\cc$ (as a $\cc$-module category) to $\mm$. This implies:

\begin{proposition} \label{prop:bmc}
The category $\mm$ is a braided $\cc$-module category. Furthermore, the braiding $e$ is uniquely determined by its restriction to simple objects of $\mm$ given by the following formula: For all $X\in \cc$ and $M=L_{\lambda,r}(=L_{\lambda}\otimes W_r)\in\mm$,
\[e_{X, M}=(c_{L_{\lambda},X}\otimes \mathrm{Id}_{W_r})\circ (c_{X,L_{\lambda}}\otimes \mathrm{Id}_{W_r}).\]
\end{proposition}

As a consequence, we observe an interesting Type BC phenomenon associated with the category $\mm$. We recall the definition of the Type BC braid group.

\begin{definition}
The Braid group associated to the Coxeter graph $BC_n$ is the group $Br_n$ generated by $n$ elements $\{\sigma_1,\sigma_2,\dots,\sigma_{n-1},\tau\}$ and satisfying the relations:
\begin{align*}
\sigma_i\sigma_{i+1}\sigma_i&=\sigma_{i+1}\sigma_i\sigma_{i+1}&&1\leq i\leq n-2\\
\sigma_i\sigma_j&=\sigma_j\sigma_i&&|i-j|>1\\
\sigma_i\tau&=\tau\sigma_i&&1\leq i\leq n-2\\
\sigma_{n-1}\tau\sigma_{n-1}\tau&=\tau\sigma_{n-1}\tau\sigma_{n-1}.
\end{align*}
\end{definition}

Then, the above proposition and \cite[Proposition 3.4]{MW26} imply that:

\begin{corollary}
Let $X\in\cc$ and $M\in \mm$. Then, for any $n\geq 2$, there exists an action of the group $Br_n$ on the space $X^{\otimes n} \rhd M$. Explicitly, the map:
\begin{align*}
\rho_n^{X,M}:Br_n&\to Aut_{\mm}(X^{\otimes n} \rhd M)\\
\sigma_i&\mapsto \mathrm{Id}_{X^{\otimes (i-1)}}\otimes c_{X,X}\otimes \mathrm{Id}_{X^{\otimes(n-i-1)}\rhd M}\\
\tau&\mapsto \mathrm{Id}_{X^{\otimes (n-1)}}\otimes e_{X,M},
\end{align*}
induces a group homomorphism.
\end{corollary}

\begin{remark}
The results of this subsection give a further example of the fact that the combinatorics appearing in the mirabolic setting are a mix of type A and type B/C. Another example is that the relations \eqref{eq:mun.4} and \eqref{eq:mun.5} imply that those generators satisfy the Type B/C quantum Serre relations
\[e_1^3\ell-[3]_ve^2\ell e + [3]_v e\ell e^2- \ell e^3 = 0, \text{ and } f_1^3\ell- [3]_vf^2\ell f + [3]_vf\ell f^2 -\ell f^3 = 0.\]
See also the discussion in \cite[Rem 3.7]{R15}.
\end{remark}

\bibliographystyle{alpha}
\bibliography{mirabolic-sln-biblist}

\end{document}